\documentclass[a4paper]{amsart}

\usepackage{amsrefs}

\usepackage{etex}
\usepackage{stmaryrd}
\usepackage{hyperref}

\usepackage{txfonts,amsmath,amstext,amsthm,amscd,amsopn,verbatim,amssymb,amsfonts,mathtools,enumitem}
\usepackage{fullpage}

\usepackage{todonotes}

\usepackage{adjustbox}

\usepackage[bbgreekl]{mathbbol}

\usepackage{tikz}
\usepackage{tikz-cd}
\usetikzlibrary{matrix}
\usetikzlibrary{shapes}
\usetikzlibrary{arrows, automata}
\usetikzlibrary{calc,3d}
\usetikzlibrary{decorations,decorations.pathmorphing,decorations.pathreplacing,decorations.markings}
\usetikzlibrary{through}

\tikzset{snakeit/.style={decorate, decoration={snake, amplitude=.2mm,segment length=1mm}}}

\tikzset{ext/.style={circle, draw,inner sep=1pt}, int/.style={circle,draw,fill,inner sep=1pt},nil/.style={inner sep=1pt}}
\tikzset{cross/.style={path picture={ 
    \draw[black]
  (path picture bounding box.south east) -- (path picture bounding box.north west) (path picture bounding box.south west) -- (path picture bounding box.north east);
  }}, uv/.style={circle, draw,cross, inner sep=1pt}}
\tikzset{cy/.style={circle,draw,fill,inner sep=2pt},scy/.style={circle,draw,inner sep=2pt},scyx/.style={draw,cross out,inner sep=2pt},scyt/.style={draw,regular polygon,regular polygon sides=3,inner sep=0.95pt}}
\tikzset{exte/.style={circle, draw,inner sep=3pt},inte/.style={circle,draw,fill,inner sep=3pt}}
\tikzset{diagram/.style={matrix of math nodes, row sep=3em, column sep=2.5em, text height=1.5ex, text depth=0.25ex}}
\tikzset{diagram2/.style={matrix of math nodes, row sep=0.5em, column sep=0.5em, text height=1.5ex, text depth=0.25ex}}
\tikzset{rowcolsep/.style={column sep=.2cm, row sep=.1cm}}

\tikzset{
  crossed/.style={
    decoration={markings,mark=at position .5 with {\arrow{|}}},
    postaction={decorate},
    shorten >=0.4pt}}

\tikzset{every picture/.style={baseline=-.65ex} }
\tikzset{every loop/.style={draw}}

  \tikzset{->-/.style={decoration={
    markings,
    mark=at position .5 with {\arrow{>}}},postaction={decorate}}}

\theoremstyle{plain}

\newtheorem{thm}{Theorem}[section]

\newtheorem{defn}[thm]{Definition}
\newtheorem{prop}[thm]{Proposition}
\newtheorem{cor}[thm]{Corollary}
\newtheorem{lemma}[thm]{Lemma}

\theoremstyle{definition}

\newtheorem{ex}[thm]{Example}
\newtheorem{rem}[thm]{Remark}

\hypersetup{pdfauthor={Thomas Willwacher}}

\newcommand{\Q}{{\mathbb{Q}}}

\newcommand{\vspan}{\mathrm{span}_\Q}

\newcommand{\flD}{\mathsf{fD}}
\newcommand{\PaRB}{\mathsf{PaRB}}

\newcommand{\dgVect}{\dg\Vect}

\newcommand{\mF}{\mathcal{F}}

\newcommand{\bbS}{\mathbb{S}}

\newcommand{\BV}{\mathsf{BV}}

\newcommand{\MMP}{Y}
\newcommand{\tMMP}{\tilde{Y}}

\renewcommand{\Bar}{{\mathtt{B}}}

\newcommand{\kk}{\mathfrak{K}}

\newcommand{\G}{\mathrm{G}}

\newcommand{\bpm}{\begin{pmatrix}}
\newcommand{\epm}{\end{pmatrix}}

\newcommand{\MC}{\mathsf{MC}}

\newcommand{\Op}{\mathcal{O}\mathit{p}}
\newcommand{\La}{\Lambda}

\newcommand{\Hopf}{\mathcal{H}\mathit{opf}}
\newcommand{\cycHOpc}{\mathcal{C}\mathit{yc}\Hopf\Op^c}
\newcommand{\PairHOpc}{\Hopf\PairOp^c}
\newcommand{\ptPairHOpc}{\pt\PairHOpc}

\newcommand{\HOpc}{\Hopf\Op^c}

\newcommand{\Free}{\mathbb{F}}

\newcommand{\Seq}{\mathcal{S}\mathit{eq}}
\newcommand{\CSeq}{\mathcal{C}\mathit{yc}\Seq}

\newcommand{\Map}{\mathrm{Map}}

\newcommand{\Aut}{\mathrm{Aut}}

\newcommand{\gr}{\mathrm{gr}}
\newcommand{\dg}{\mathit{dg}}

\newcommand{\GRT}{\mathrm{GRT}}

\newcommand{\Vect}{{\mathcal{V}\mathit{ect}}}

\newcommand{\oW}{\mathring{W}} 

\newcommand{\beq}[1]{\begin{equation}\label{#1}}
\newcommand{\eeq}{\end{equation}}

\newcommand{\dgca}{\mathrm{dgca}}

\newcommand{\MP}{X}

\newcommand{\mG}{{\mathcal G}}

\newcommand{\PaCD}{\mathsf{ PaCD}}

\newcommand{\eis}{\mathrm{1}}
\newcommand{\peis}{\mathrm{1}_{Pair}}
\newcommand{\meis}{\mathrm{1}_M}

\newcommand{\SO}{\mathrm{SO}}

\newcommand{\Res}{\mathrm{Res}}
\newcommand{\Ind}{\mathrm{Ind}}

\newcommand{\ar}{\mathrm{ar}}
\newcommand{\Mod}{\mathcal{M}\mathit{od}}
\newcommand{\ptMod}{\mathrm{pt}\Mod}
\newcommand{\Modc}{\Mod^c}
\newcommand{\dgModc}{\dg\Mod^c}
\newcommand{\dgHModc}{\Hopf\Mod^c}

\newcommand{\iHom}{{\mathcal{H}om}}

\newcommand{\sset}{\mathit{s}\mathcal{S}\mathit{et}}

\newcommand{\sLie}{\mathsf{sLie}}

\newcommand{\fg}{\mathfrak{g}}
\newcommand{\fh}{\mathfrak{h}}

\newcommand{\cone}{\mathrm{cone}}
\newcommand{\sSet}{\sset}

\newcommand{\FreeMod}{\Free}

\newcommand{\M}{\mathcal{M}}

\DeclareMathAlphabet{\mathsfit}{OT1}{cmss}{m}{sl}
\DeclareMathOperator{\AOp}{\mathsfit{A}}
\DeclareMathOperator{\BOp}{\mathsfit{B}}
\DeclareMathOperator{\COp}{\mathsfit{C}}
\DeclareMathOperator{\DOp}{\mathsfit{D}}

\DeclareMathOperator{\MOp}{\mathsfit{M}}
\DeclareMathOperator{\NOp}{\mathsfit{N}}
\DeclareMathOperator{\POp}{{\mathsfit{P}}}
\DeclareMathOperator{\QOp}{{\mathsfit{Q}}}

\DeclareMathOperator{\dgOpc}{\dg\Op^c}

\DeclareMathOperator{\dgOp}{\dg\Op}

\DeclareMathOperator{\dgHOpc}{\dg\Hopf\Op^c}

\newcommand{\eql}{\mathrm{eq}}
\newcommand{\coeql}{\mathrm{coeq}}
\newcommand{\cC}{{\mathcal C}}
\newcommand{\cycOp}{{\mathcal{C}\mathrm{yc}\Op}}

\newcommand{\PairOp}{{\mathcal{P}\mathit{air}}}
\newcommand{\ptPairOp}{\pt\PairOp}
\newcommand{\pt}{{\mathit{pt}}}

\newcommand{\ftc}{\mathfrak t^c}
\newcommand{\ft}{\mathfrak t}
\newcommand{\frtc}{\mathfrak {ft}^c}
\newcommand{\frt}{\mathfrak {ft}}
\newcommand{\PaP}{\mathsf{PaP}}
\newcommand{\tPaP}{\widetilde{\PaP}}
\newcommand{\PaRCD}{\mathsf{PaRCD}}

\newcommand{\dgCycOp}{\dg\cycOp}
\newcommand{\dgptPairOp}{\dg\pt\PairOp}
\newcommand{\dgPairOp}{\dg\PairOp}
\newcommand{\dgptMod}{\dg\pt\Mod}
\newcommand{\dgMod}{\dg\Mod}

\author{Thomas Willwacher}
\address{Department of Mathematics \\ ETH Zurich \\
R\"amistrasse 101 \\
8092 Zurich, Switzerland}
\email{thomas.willwacher@math.ethz.ch}
\thanks{The author has been partially supported by the NCCR Swissmap, funded by the Swiss National Science Foundation}

\begin{document}
\title{Cyclic operads through modules}
\begin{abstract}
We describe a way to compute mapping spaces of cyclic operads through modules. As an application we compute the homotopy automorphism space of the cyclic Batalin-Vilkovisky (Hopf co-)operad.
\end{abstract}
\maketitle

\section{Introduction}

Let $\POp$ be a (unital augmented) cyclic dg operad. Then we can associate to $\POp$ the non-cyclic operad $\POp^{nc}:=\POp$, and the right $\POp^{nc}$-module $\POp^{mod}=\POp$.
Hence we obtain a "forgetful" functor 
\[
F : \dg\cycOp \to \dgptPairOp    
\]
from the category of cyclic operads to the category of pairs $(\QOp, \MOp)$ consisting of a non-cyclic (unital augmented) operad $\QOp$, and a right pointed $\QOp$-module $\MOp$. Here "pointed" means that there is a distinguished element $1\in \MOp((2))^{S_2}$ and an augmentation, that remember the unit and augmentation on $\POp$.

The purpose of this paper is to show that $F$ is homotopically fully faithful, or more precisely:
\begin{thm}\label{thm:main dg}
The forgetful functor $F$ is part of a Quillen adjunction 
\[
G \colon  \dgptPairOp  \rightleftarrows \dgCycOp \colon F.   
\]
For $\POp$ any augmented cyclic operad the derived counit of the adjunction 
\[
LG( RF (\POp)) \to \POp    
\] 
is a weak equivalence.
\end{thm}
It follows that for $\POp$, $\QOp$ augmented cyclic operads we have a weak equivalence of the derived mapping spaces 
\[
\Map^h_{\dgCycOp}(\POp, \QOp) \simeq  
\Map^h_{\dgptPairOp}(F(\POp), F(\QOp)).  
\]
The right-hand side has two benefits over the left-hand side.
First, to compute the right-hand side we need models for $\POp$ and $\QOp$ not as cyclic operads, but only as non-cyclic operads and modules over itself.
Second, the right-hand side helps to understand the difference between cyclic operadic mapping spaces and ordinary operadic mapping spaces for cyclic operads. 
Concretely, one has the natural comparison map
\begin{equation} 
    \label{equ:cyc nc comparison}
    \begin{tikzcd}
    \Map^h_{\dgptPairOp}(F(\POp), F(\QOp)) \ar{r}
    & 
    \Map^h_{\dgOp}(\POp^{nc}, \QOp^{nc})
    \end{tikzcd},
\end{equation}
projecting the morphism of pairs to its (non-cyclic) operadic part.
The morphism \eqref{equ:cyc nc comparison} is a fibration (for suitable models of the derived mapping spaces), and the fiber over $f:\POp^{nc}\to \QOp^{nc}$ is the pointed module mapping space 
$\Map^h_{\dgptMod_{\POp^{nc}}}(\POp^{mod}, f^*\QOp^{mod})$.
From this we immediately obtain a criterion for a non-cyclic operad map between cyclic operads to be homotopic to a cyclic operad map.
\begin{cor}
Let $\POp$, $\QOp$ be augmented cyclic dg operads with $\POp$ cofibrant and let 
$f:\POp^{nc}\to \QOp^{nc}$ be a morphism of non-cyclic dg operads.
Then $f$ is homotopic to a cyclic operad morphism if and only if there is a morphism of pointed $\POp$-modules $\POp^{mod}\to f^*\QOp^{mod}$. 
\end{cor}

Furthermore, we obtain a long exact sequence of homotopy groups 
\[
\cdots \to 
\pi_k \Map^h_{\dgptMod_{\POp}}(\POp^{mod}, f^*\QOp^{mod})  
\to 
\pi_k \Map^h_{\dgCycOp}(\POp, \QOp)
\to 
\pi_k \Map^h_{\dgOp}(\POp^{nc}, \QOp^{nc})    
\to \cdots.
\]
that allows us to compare the cyclic and non-cyclic mapping spaces.

For the purposes of rational homotopy theory, one is typically interested not in dg operads, but in dg Hopf cooperads, that is, cooperads in differential graded commutative algebras.
The version of Theorem \ref{thm:main dg} for dg Hopf cooperads reads:

\begin{thm}\label{thm:main Hopf}
Let $F:\cycHOpc \to \pt\PairHOpc$ be the forgetful functor sending a cyclic dg Hopf cooperad $\COp$ to the pair $(\COp^{nc},\COp^{mod})$ consisting of the non-cyclic dg Hopf cooperad $\COp^{nc}:=\COp$ and the pointed $\COp^{nc}$-comodule $\COp^{mod}:=\COp$. Then $F$ is part of a Quillen adjunction 
\[
F \colon \cycHOpc  \rightleftarrows  \ptPairHOpc \colon G, 
\]
and for any cofibrant cyclic dg Hopf cooperad $\COp$ the derived unit of the adjunction 
\[
    \COp \to  RG(LF(\COp))    
\] 
is a weak equivalence.
\end{thm}

Again this allows us to compare the mapping spaces of non-cyclic and cyclic dg Hopf cooperads analogous to the dg operadic situation.
We apply our results to one particular example.
The Batalin-Vilkovisky dg Hopf cooperad 
\[
\BV^c := H^\bullet(\flD_2),
\]
is defined as the cohomology cooperad of the framed little disks operad $\flD_2$. It is well-known \cite{GKcyclic, BudneyCyclic} that $\BV^c$ is a cyclic Hopf cooperad.
The homotopy automorphism simplicial monoid of the non-cyclic cooperad $\BV^{c,nc}$ has been computed in \cite{HWframed} to be 
\begin{equation}\label{equ:pi aut bv}
    \Aut^h_{\HOpc}(\BV^{c,nc}) 
    \simeq
        \GRT \ltimes \SO(2)^\Q,
\end{equation}
with $\GRT$ the Grothendieck-Teichmüller group.
We show in Proposition \ref{prop:GRT action} below that the $\GRT$-action 
in fact preserves the cyclic structure on $\BV^c$.
Furthermore, essentially by a degree counting argument one can show:
\begin{thm}\label{thm:aut BVc contr}
    The homotopy automorphism group of the $\BV^{c,nc}$-Hopf-comodule $\BV^{c,mod}$ is weakly contractible, i.e., 
\[
   \pi_k \Aut^h_{\dgHModc_{\BV^{c,nc}}}(\BV^{c,mod})
   =0 \quad \text{for all $k$.}
\] 
\end{thm}
Note that this result is about comodule automorphisms, not pointed comodule automorphisms. However, they can be compared and related to each other and one eventually finds:
\begin{cor}\label{cor:Aut BV}
The homotopy automorphism group of the Batalin-Vilkovisky cooperad $\BV^c$ as a cyclic dg Hopf cooperad is weakly equivalent to the Grothendieck-Teichmüller group $\GRT$, considered as a discrete group, that is
\[
    \pi_k
    \Aut^h_{\cycHOpc}(\BV^c) 
    \cong 
    \begin{cases}
     \GRT & \text{for $k=0$} \\
     0  &  \text{for $k\geq 1$}
    \end{cases} .
\]
\end{cor}

\section{Prerequisites and model categories}
\subsection{Conventions on dg vector spaces}
We use cohomological conventions throughout, so that all our differentials have degree +1. All vector spaces will be $\Q$-vector spaces.
We generally work in either the category $\dgVect_{\geq 0}$ of non-negatively graded dg vector spaces or the category $\dgVect_{\leq 0}$ of non-positively graded dg vector spaces.
We equip these categories with the following (standard) model category structures:
\begin{itemize}
\item The weak equivalences in either category are the quasi-isomorphisms.
\item The cofibrations in $\dgVect_{\geq 0}$ are the morphism that are injective in positive degrees.  The fibrations in $\dgVect_{\geq 0}$ are the morphism that are surjective in all degrees. 
\item The cofibrations in $\dgVect_{\leq 0}$ are the morphism that are injective in all degrees.  The cofibrations in $\dgVect_{\leq 0}$ are the morphism that are surjective in negative degrees. 
\end{itemize}

\subsection{Symmetric sequences and cyclic sequences}
\label{sec:symseq}
A symmetric sequence $\AOp$ in a category $\cC$ is a collection of objects $\AOp(r)$ with a right action of the symmetric group $S_r$ for each $r=1,2,\dots$.
We write $\Seq\cC$ for the category of symmetric sequences.
A cyclic sequence $\BOp$ is the same data, but starting at $r=2$ instead of $r=1$. The usual convention is to write $\BOp((r))$ for the $r$-ary part. We write $\CSeq\cC$ for the category of cyclic sequences.
There is an obvious forgetful functor
\begin{gather*}
\Res :  \CSeq\cC \to \Seq\cC \\
(\Res(\BOp))(r) = \Res_{S_{r+1}}^{S_r}\BOp((r+1)).
\end{gather*}
In the situations of interest to us, in particular for $\cC$ the category of dg vector spaces, this functor has a left adjoint that we denote by 
\begin{gather*}
    \Ind :  \Seq\cC \to \CSeq\cC \\
    (\Ind(\AOp))((r)) = \Ind_{S_{r-1}}^{S_r}\AOp((r-1)).
\end{gather*}

It is often convenient to index cyclic or symmetric sequences by finite sets instead of numbers, i.e., we may define a symmetric sequence as a contravariant functor from the category of finite sets with bijections as morphisms to the category of dg vector spaces.
Then $\POp((n)):=\POp(\{1,\dots,n\})$ is the value of the functor on the set $\{1,\dots,n\}$, and conversely one may recover the functor on a set $A$ by setting
\[
\POp((A)) := \left(\bigoplus_{f: A\xrightarrow{\cong}\{1,\dots,n\}} \POp((n))  \right)_{S_rn},
\]
where the direct sum is over bijections from $A$ to $\{1,\dots,n\}$ and the symmetric group acts diagonally on the set of such bijections and on $\POp((n))$.
Intuitively, this corresponds to labeling the inputs of our operations by elements of $A$ rather than numbers.
We shall freely pass between both definitions of symmetric sequences and always use the most convenient.

\subsection{Model categories of operads and cooperads}

\subsubsection{Dg operads}
We denote by $\dgOp$ (respectively $\dgCycOp$) the category of augmented (resp. augmented cyclic) dg operads.
We assume that the underlying dg vector spaces are non-positively cohomologically graded.
We also assume that our operads do not have operations of arity zero.

The category of pairs $(\POp, \MOp)$ of a dg operad $\POp$ and an operadic right $\POp$-module $\MOp$ is denoted by $\dgPairOp$.
The category of right operadic $\POp$-modules, for a fixed operad $\POp$, is denoted by $\dgMod_{\POp}$.
As usual, we equip all these categories with cofibrantly generated model structures by transfer along the forgetful functors to dg symmetric sequences
\begin{align*}
    \dgOp &\to \Seq\dgVect_{\leq 0}
    &
    \dg\cycOp &\to \CSeq\dgVect_{\leq 0}
    \\
    \dg\PairOp &\to \Seq\dgVect_{\leq 0} \times \Seq\dgVect_{\leq 0}
    &
    \dg\Mod_{\POp} &\to \Seq\dgVect_{\leq 0}
      .
\end{align*}
More precisely, the forgetful functor $\dgOp \to \Seq\dgVect_{\leq 0}$ associates to an augmented dg operad $\POp$ with augmentation $\epsilon : \POp \to 1$ the augmentation ideal $\bar P=\ker\epsilon$. The other forgetful functors are defined similarly.
The transfered model structure then has the following distinguished classes of morphism:
\begin{itemize}
\item The weak equivalences are the quasi-isomorphisms.
\item The cofibrations are the morphism are the morphisms that are arity- and degreewise injective maps. 
\item The fibrations are the morphisms that are aritywise surjective maps in negative degrees.
\end{itemize}
We refer to \cite[Theorem 2.1]{BMColored} for the fact that the right-transfer yields well-defined cofibrantly generated model category structures in these cases. 

We define the special object of $\dg\PairOp$
\[
\peis = (\eis, \meis)
\]
with 
\begin{equation}\label{equ:meis def}
\meis(r)=
\begin{cases}
    \Q & \text{for $r=2$} \\
    \emptyset & \text{otherwise}
\end{cases},
\end{equation}
with $\Q$ considered as the trivial $S_2$-module concentrated in cohomological degree zero.
The object $\peis$ is fibrant and cofibrant in $\dg\PairOp$. We define the over-under-category
\[
    \dgptPairOp := \dgPairOp^{\peis/}_{/\peis},
\]
whose objects are factorizations of the identity
\[
\peis \to (\POp, \MOp) \to \peis.  
\]
We equip $\dgptPairOp$  with the slice model category structure, see \cite{HirschhornSlice}. That is, the weak equivalences (resp. fibrations cofibrations) are those morphisms that are weak equivalences (resp. fibrations, cofibrations) in $\dgPairOp$.

Similarly, for $\POp$ a fixed dg operad, we consider the free right $\POp$-module 
\[
    \eis_{\POp} = \FreeMod_{\POp}(\meis).
\]
This is again cofibrant by freeness. We equip the over-under-category 
\[
    \dgptMod_{\POp} := (\dgMod_{\POp})^{\eis_{\POp}/}_{/\meis}      
\]
with the slice model structure as well.

\subsubsection{Dg cooperads}
Let $\dgOpc$ (resp. $\dgCycOp^c$) be the category of non-negatively graded coaugmented, conilpotent dg cooperads (resp. cyclic cooperads).
For $\COp\in \dgOpc$ we denote the category of conilpotent right $\COp$-comodules by $\Modc_{\COp}$.
Let $\dgPairOp^c$ be the category of pairs $(\COp, \MOp)$, with $\COp\in \dgOpc$ and $\MOp\in \Modc_{\COp}$.
There are cofree/forgetful adjunctions 
\begin{align*}
    \dgOpc &\rightleftarrows \Seq\dgVect_{\geq 0}
    &
    \dg\cycOp^c &\rightleftarrows \CSeq\dgVect_{\geq 0}
    \\
    \dg\PairOp^c &\rightleftarrows \Seq\dgVect_{\geq 0} \times \Seq\dgVect_{\geq 0}
    &
    \dg\Modc_{\COp} &\rightleftarrows \Seq\dgVect_{\geq 0}
      .
\end{align*}
Here the right-adjoints are defined by the cofree cooperad (resp. cyclic cooperad, comodule) functors. The left-adjoints are the forgetful functors -- just mind that since we work with coaugmented cooperads the forgetful functors take the coaugmentation coideal.
We may define model category structures on all categories above by left transfer along the above adjunctions.
For the case of $\dgOpc$ it has been verified in \cite{Frextended} that this yields a well-defined cofibrantly generated model category structure. The results are extended to pairs and comodules in \cite{FWColored}.
Finally, the case of cyclic operads cannot be found in the literature, but can be extracted from \cite{Frextended} by just replacing cooperads by cyclic cooperads.

We also define pointed comodules and pointed pairs, analogously to the previous section.
We consider the following object of $\dgPairOp^c$
\[
\peis^* = (\eis^*, \meis^*)    
\]
with
\[
\meis^*((r))=
\begin{cases}
    \Q & \text{for $r=2$} \\
    \emptyset & \text{otherwise}
\end{cases},
\]
with $\Q$ considered with trivial $S_2$-action.
The object $\peis^*$ is cofree and hence fibrant in $\dgPairOp^c$.
We define the category of pointed comodule pairs as the over-under-category 
\[
\dgptPairOp^c =  (\dgPairOp^c)_{/ \peis^*}^{\peis^*/}
\]
with the slice category model structure. We also define, for a fixed dg Hopf cooperad $\COp$, the cofree right $\COp$-comodule
\[
\eis_{\COp}^* := \FreeMod^c_{\COp} \meis.    
\] 
Then we equip the over-under-category 
\[
    \dgptMod^c_{\COp} := (\dgModc_{\COp})_{/ \peis^*}
\]
with the slice category model structure.

\subsubsection{Dg Hopf cooperads}
Let $\HOpc$ be the category of (non-cyclic, conilpotent) dg Hopf cooperads, i.e., cooperads in the underlying category of dg commutative algebras. Similarly, let
$\cycHOpc$ be the category of cyclic dg Hopf cooperads.
For a dg Hopf cooperad $\COp$ denote the category of right Hopf $\COp$-comodules by $\dgHModc_{\COp}$.
Let $\PairHOpc$ the category of pairs $(\COp, \MOp)$ consisting of a non-cyclic dg Hopf cooperad $\COp$ and a right $\COp$ dg Hopf comodule $\MOp$.
One has adjunctions
\begin{equation}\label{equ:hopf forget adj}
\begin{aligned}
    \dgOpc &\rightleftarrows \HOpc
    &
    \dg\cycOp^c &\rightleftarrows \cycHOpc
    \\
    \dgPairOp^c &\rightleftarrows \PairHOpc
    &
    \dg\Modc_{\COp} &\rightleftarrows \dgHModc_{\COp},
\end{aligned}
\end{equation}
with the left-adjoint the forgetful functor forgetting the Hopf (i.e., commutative algebra) structure, and the right adjoint the arity-wise symmetric algebra functor, see \cite[section 1.5]{Frextended} and \cite[II.9.3]{Frbook}. 
It has been shown by Fresse \cite{Frextended} that by right transfer along the above adjunction one may define a cofibrantly generated model structure on $\HOpc$. Fresse's construction readily extends to the cyclic setting to define a model structure on $\cycHOpc$.
It also extends to colored dg Hopf cooperads, in particular to pairs, and can be used to endow $\dgHModc_{\COp}$ and $\PairHOpc$ with cofibrantly generated model structures, see \cite{FWColored}.

We also define define the corresponding pointed versions.
Note that the object $\peis^*$, $\meis^*$ and $\eis_{\COp}^*$ of the previous subsection carry natural arity-wise dgca structures and can be considered as objects in the respective category of dg Hopf objects.
We then define the category of pointed Hopf pairs as the over-category 
\[
\ptPairHOpc =  \PairHOpc_{/ \peis^*}
\]
with the slice category model structure. Next, fix a dg Hopf cooperad $\COp$. THen we equip the over-category 
\[
    \pt\dgHModc_{\COp} := (\dgHModc_{\COp})_{/ \peis^*}
\]
with the slice category model structure.
Note that for Hopf cooperads we do not need to work with over-under-categories, since the coaugmentation is implicitly part of the Hopf structure, in the form of the inclusion of the commutative algebra units.


\subsection{Fiber sequences for operad-module pairs}

\begin{prop}\label{prop:pair op op fibration}
Let $(\POp, \MOp)$ be a cofibrant object of $\dgPairOp$ or $\dgptPairOp$, and let $(\QOp, \NOp)$ be a fibrant object of the same category. Then the projection to the operadic part 
\begin{equation}\label{equ:pair op op fibration}
\Map_{\dg(\pt)\PairOp}\left((\POp, \MOp), (\QOp, \NOp)\right)  
\to 
\Map_{\dg\Op}\left(\POp, \QOp \right)  
\end{equation}
is an $\sSet$-fibration. The fiber over a morphism $f:\POp\to \QOp$ is 
\[
    \Map_{\dg(\pt)\Mod_{\POp}}\left(\MOp, f^*\NOp \right).
\]
\end{prop}
\begin{proof}
We first note that 
\[
    \Map_{\dg(\pt)\Op}\left(\POp, \QOp \right)  
    =
    \Map_{\dg(\pt)\PairOp}\left((\POp, \MOp), (\QOp, *)\right).   
\]
Next, the map $(\QOp, \NOp)\to (\QOp, *)$ is a fibration, because $\NOp$ is fibrant by assumption.
But the functor 
\[
    \Map_{\dg(\pt)\PairOp}\left((\POp, \MOp), -\right):
\dg(\pt)\PairOp \to \sSet
\] 
is right Quillen for $(\POp, \MOp)$ a cofibrant object, see \cite[Proposition 3.2.12]{Frbook}. 
Applying this functor to the fibration $(\QOp, \NOp)\to (\QOp, *)$ we hence obtain that \eqref{equ:pair op op fibration} is a fibration.

Consider now the pointed setting, the non-pointed setting is similar.
Fix a map of augmented dg operads $f:\POp\to \QOp$ and consider the fiber over this map in the fibration \eqref{equ:pair op op fibration}. 
It is identified with the mapping space in the over-under-category
\[
\Map_{\dg\PairOp_{(\QOp,\meis)}^{\peis}}((\POp,\MOp),(\QOp, \NOp)),
\]
where we understand $(\POp,\MOp)$ as the object 
\[
\peis \to (\POp,\MOp) \xrightarrow{f} (\QOp,\meis).
\]
But a morphism $(\POp,\MOp) \to (\QOp,\NOp)$ in this over-under-category is the same data as a morphism of $\POp$-modules $\phi: \MOp\to f^*\NOp$ that satisfies into a commutative diagram of $\POp$-module morphisms 
\[
\begin{tikzcd}
    & \eis_{\POp} \ar{dl}\ar{dr} & \\
    \MOp \ar{rr}{\phi}\ar{dr}& & \NOp\ar{dl} \\
    & \meis & 
\end{tikzcd}.
\]
But this is the same data as a morphism in $\dg\pt\Mod_{\POp}$.
Here we also used that for $\POp$ a fixed unital operad a morphism $\peis\to (\POp, \MOp)$ is the same data as a morphism of $1$-modules (i.e., symmetric sequences) $\meis\to \MOp$, and by adjunction also the same data as a morphism of $\POp$-modules $\eis_{\POp}\to \MOp$.

\end{proof}

The analogous result for dg Hopf cooperads also holds.
\begin{prop}\label{prop:pair opc opc fibration}
    Let $(\AOp, \MOp)$ be a cofibrant object of $\PairHOpc$ or $\pt\PairHOpc$, and let $(\BOp, \NOp)$ be a fibrant object of the same category. Then the projection to the cooperadic part 
    \begin{equation}\label{equ:pair opc opc fibration}
    \Map_{(\pt)\PairHOpc}\left((\AOp, \MOp), (\BOp, \NOp)\right)  
    \to 
    \Map_{\HOpc}\left(\AOp, \BOp \right)  
    \end{equation}
    is an $\sSet$-fibration. The fiber over a morphism $f:\AOp\to \BOp$ is 
    \[
        \Map_{(\pt)\dgHModc_{\BOp}}\left(f_*\MOp, \NOp \right).
    \]
    \end{prop}
    \begin{proof}
The proof is analogous to that of Proposition \ref{prop:pair op op fibration}.
    \end{proof}

\subsection{Adjunctions between cyclic and non-cyclic (co)operads}
\label{sec:op cycop adj}
The forgetful functor $\Res$ from cyclic operads to operads fits into an adjunction 
\[
I \colon \dg\Op \rightleftarrows \dg\cycOp \colon \Res.
\]
The left-adjoint $I$ is defined such that for a non-cyclic dg operad $\QOp$
\[
I(\QOp) = \Free_{cyc}(\Ind \bar \QOp)/\sim
\]
is the free cyclic operad generated by the coaugmentation coideal $\bar \QOp$ modulo the cyclic operadic ideal generated the following relation:
\begin{itemize}
\item Let $q,q'\in \bar \QOp$ two elements and let $\tilde q$, $\tilde q'$ be the corresponding elements in $\Ind \bar \QOp\subset I(\QOp)$. Then the cyclic composition in $I(\QOp)$ satisfies 
\[
\tilde q \circ_{i,0} \tilde q' = 
\widetilde{q\circ_i q'}.
\]
\end{itemize}

On the other hand the forgetful functor $\Res$ also has a right adjoint. (See \cite{DruColeHackney} for a discussion of both adjoints, albeit in more general setting than is relevant here.)
We will actually need the other adjoint for dg cooperads, where it fits into an adjunction 
\[
\Ind_1 \colon \dgOpc \rightleftarrows \dgCycOp^c \colon \Res.  
\]
Here the right-adjoint is again the forgetful functor. 
The left adjoint sends a coaugmented dg cooperad $\COp$ to the coaugmented cyclic dg cooperad obtained by adjoining a counit to the non-counital cyclic cooperad $\overline{\Ind_1(\COp)}$ defined such that
\[
    \overline{\Ind_1(\COp)}((r))
    = (\Ind \bar\COp)((r))
    = \bar \COp(r-1) \otimes_{\bbS_{r-1}} \Q[\bbS_r].
\]
In other words, we apply the induction functor $\Ind$ from section \ref{sec:symseq} aritywise to $\bar C$.

\section{Bar constructions}

\subsection{Cyclic bar and cobar construction}
A one-shifted cooperad is a cooperad whose cocomposition has cohomological degree $+1$.
Let $\POp$ be an augmented dg operad. Then we define the bar construction of $\POp$ as the 
cofree 1-shifted cooperad cogenerated by the augmentation ideal $\bar \POp$,
\[
\Bar \POp = (\Free_{\kk}^c \bar \POp, D),    
\]
with a differential $D$ encoding the dg operad structure on $\POp$.
More concretely, elements of $\Free_{\kk}^c \bar \POp$ can be seen as linear combinations of trees whose vertices are decorated by copies of $\bar \POp$. Explicitly, we have the formula 
\begin{equation}\label{equ:bar def}
\Free_{\kk}^c \bar \POp(r) = \bigoplus_{T} (\otimes_{e\in ET} \Q[1])\otimes (\otimes_T \bar \POp)  ,  
\end{equation}
where the sum is over all rooted trees with $r$ numbered leaves, $\otimes_T$ is the tree-wise tensor product and the final tensor product runs over the internal edges of $T$. Note that the last factor only contributes an overall degree shift.
The cooproduct is by de-grafting (cutting) trees. This operation is of degree $+1$ since one edge is removed, which carries degree $-1$ by the last factor in the formula above.

The differential $D$ on $\Bar \POp$ has two terms, $D=d_{\bar \POp}+d_c$, with $d_{\bar \POp}$ induced by the internal differential on $\POp$, while $d_c$ acts on a decorated tree by contracting one edge, operadically composing the decorations at the adjacent vertices.

\begin{rem}
We remark that we deviate here from the standard convention to define the bar construction as the cofree cooperad cogenerated by $\bar \POp[1]$. Our convention differs just by an overall degree shift.
\end{rem}
Let $\COp$ be a 1-shifted coaugmented cooperad. Then we define the cobar construction of $\COp$ as the free operad generated by the coaugmentation coideal $\bar \COp$ 
\[
    \Bar^c \COp = (\Free \bar \COp, D), 
\]
with a differential $D$ encoding the dg cooperad structure on $\COp$.

Similarly, let $\POp$ be an augmented cyclic dg operad. Then the bar construction of $\POp$ is the cofree 1-shifted cyclic cooperad cogenerated by the augmentation ideal $\bar \POp$,
\[
\Bar \POp = (\Free_\kk^c \bar \POp, D),    
\]
with a differential $D$ encoding the dg cyclic operad structure on $\POp$ as above.
We use the same notation as for the non-cyclic bar construction, since both are isomorphic.
That said, we also use the more verbose notation $\Bar_{cyc} \POp=\Bar\POp$ for the same object if needed to avoid confusion.

The cobar construction of a 1-shifted coaugmented dg cyclic cooperad $\COp$ is the free cyclic operad generated by $\bar \COp$, 
\[
\Bar^c \COp = (\Free \bar \COp, D),    
\]
with a differential $D$ encoding the dg cyclic cooperad structure on $\COp$.

In either case one has a natural quasi-isomorphism 
\[
    \Bar^c \Bar \POp \to \POp 
\]
for any (cyclic) augmented dg operad $\POp$. Furthermore, $\Bar^c \Bar \POp$ is a cofibrant object in the category $\dgOp$ or $\dgCycOp$ respectively, see \cite[Theorem 2]{DehlingVallette}. We also apply the bar and cobar construction to non-unital (co)operads. In this case we just do not pass to the (co)augmentation (co)ideals in the formulas above.

The analogous constructions also work for colored operads.
In particular, a pair $(\POp,\MOp)$ consisting of a (non-cyclic, augmented) operad $\POp$ and a right $\POp$-module $\MOp$ can be considered as a two-colored operad, with $\POp$ the operations of input and output color 1, and $\MOp$ the operations of input color 1 and output color 2.
Their colored operadic bar construction is again a colored 2-colored 1-shifted cooperad associated to the pair 
\[
\Bar(\POp,\MOp) = (\Bar \POp, \Bar_{\POp} \MOp),    
\]
with 
\[
\Bar_{\POp} \MOp =(\Free_{\Bar \POp}^c \MOp, D)
\]
the cofree $\Bar \POp$-comodule cogenerated by $\MOp$, with a differential encoding the $\POp$-module structure on $\MOp$.

\subsection{The cobar construction often preserves weak equivalences}

\begin{lemma}\label{lem:cyclic cobar}
    Let $F: \COp\to \DOp$ be a morphism of 1-shifted cyclic dg cooperads.
    Assume that $\bar\COp$. $\bar\DOp$ are equipped with exhaustive ascending filtrations 
    \begin{align*}
    0 &= \mF^0\bar\COp \subset \mF^1\bar\COp \subset \cdots
    &
    0 &= \mF^0\bar\DOp \subset \mF^1\bar\DOp \subset \cdots
    \end{align*}
    compatible with the cooperad structure and the morphism $F$.
    Suppose that the morphism $F$ induces a quasi-isomorphism on the associated graded dg symmetric sequences with respect to the above filtrations
    \[
      F_0: \gr \bar \COp \to \gr \bar \DOp.
    \]
    Then the morphism 
    \[
    F: \Bar^c(\COp) \to \Bar^c(\DOp)    
    \]
    is a quasi-isomorphism of cyclic operads.
    \end{lemma}
    \begin{proof}
        We equip $\Bar^c(\COp)$ and $\Bar^c(\DOp)$ with the exhaustive ascending filtrations induced by the filtrations $\mF^\bullet\COp$ and $\mF^\bullet\DOp$.
        It then suffices to check that the induced map on the associated graded objects
        \[
        \gr_{\mF} F :\gr_{\mF} \Bar^c(\COp) \to \gr_{\mF}\Bar^c(\DOp) 
        \]
        is a quasi-isomorphism of dg symmetric sequences.
        The differential on the associated graded has the form $d_{\COp}+d_{split}$ (resp. $d_{\DOp}+d_{split}$) with the first term induced by the internal differentials on $\COp$ and $\DOp$ and the second term given by the cooperadic cocomposition.
        We filter the associated graded in turn with a descending filtration on the number of vertices in the trees appearing in the free operad construction.
        Concretely, we have that 
        \begin{align*}
        &\mG^p \gr_{\mF}^q\Bar^c(\COp) & &\mG^p\gr^q_{\mF}\Bar^c(\DOp) 
        \end{align*}
        are the subspaces of linear combinations of trees with $\geq p$ vertices.
        The important point is now that for each fixed $q$ the number of possible such trees is finite due to the properties of the filtration $\mF$.
        Hence the filtration $\mG$ is bounded and the associated spectral sequence converges.
        By the assumption on $F$ and Künneth formula we hence see that $F$ induces a quasi-isomorphism on the associated graded complexes 
        \[
            \gr_{\mF}^p \gr_{\mF}^q\Bar^c(\COp) 
            \to 
            \gr_{\mF}^p \gr_{\mF}^q\Bar^c(\DOp).
        \] 
        Hence by the spectral sequence comparison theorem (applied twice) we conclude that $F$ is indeed a quasi-isomorphism.
        \end{proof}
\begin{ex}
    If $\bar \COp(1)$ and $\bar \DOp(1)$ are concentrated in degrees $\leq -1$, then the filtration in the Lemma may be taken to be by arity minus the cohomological degree minus one.
    \end{ex}
    
    \begin{ex}
    Another example is that $\COp=\Bar\POp$, $\DOp=\Bar\QOp$ are bar constructions. Then the filtrations in the lemma may be taken to be by the numbers of vertices in the trees in the definition of the bar construction.
    \end{ex}

\subsection{Auxiliary symmetric sequences}
\label{sec:auxseq}
We define an auxiliary dg cyclic sequence $\MP$ such that
\[
\MP((r)) = \vspan\{c , \partial_1,\dots, \partial_r\}    
\]
is spanned by symbols $\partial_j$ of cohomological degree $-1$ and a symbol $c$ of degree $0$.
The action of $\bbS_r$ is by permuting the $\partial_j$ and trivial on $c$.
The differential $\partial$ is defined such that 
\begin{align*}
    \partial \partial_j &= c 
    &
    \partial c &= 0.
\end{align*}
We also define the cyclic sequences 
\begin{align*}
\MP'((r)) &:= \vspan(\partial_1,\dots, \partial_r)
\subset \MP((r))
\\
\MMP((r)) &:= H(\MP((r)))
\cong \vspan\{\partial_1-\partial_2,\partial_1-\partial_3,\dots, \partial_1- \partial_r\} \subset \MP'((r)).
\end{align*}

These cyclic sequences have the following meaning.
For a dg cyclic sequence $\COp$ we consider the counit of the induction/restriction adjunction of section \ref{sec:symseq},
\[
\Ind(\Res\COp) \to \COp.    
\]
Then the cyclic sequence
$$
\Ind(\Res\COp) \cong  \MP'[-1]\otimes \COp
$$ 
can be identified with the arity-wise tensor product with the degree shifted cyclic sequence $\MP'$.
Furthermore, we have that 
$$
\MP\otimes \COp \cong\cone(\Ind(\Res\COp) \to \COp) 
$$
can be identified with the mapping cone of the counit morphism.

\subsection{Comparison of bar constructions}
Let $\POp$ be an augmented cyclic dg operad.
Then we have defined two different bar constructions for $\POp$: The (cyclic) operadic bar construction $\Bar\POp$ of $\POp$ and the module bar construction $\Bar\POp^{mod}$ of $\POp^{mod}$ as a right $\POp$-module.
Our next goal is to compare these two constructions.
Both $\Bar\POp$ and $\Bar\POp^{mod}$ have natural coaugmentations, and it is sufficient to compare the coaugmentation coideals $\overline{\Bar\POp}$ and $\overline{\Bar\POp^{mod}}$.

\begin{prop}\label{prop:MMPs}
    For $\POp$ an augmented cyclic dg operad there is a chain of quasi-isomorphisms of dg symmetric sequences
    \[
        \overline{\Bar\POp^{mod}}
        \xrightarrow{f} 
        \MP \otimes \overline{\Bar\POp}  
    \leftarrow 
    \MMP \otimes
    \overline{\Bar\POp},
    \]
    with $\MP$ and $\MMP$ the symmetric sequences of section \ref{sec:auxseq}
\end{prop}

\begin{proof}
It is clear that the second arrow is a quasi-isomorphism since the morphism $\MMP\to\MP$ is. We just need to construct the first map $f$ and show that it is a quasi-isomorphism as well. 

To this end recall that elements of $\overline{\Bar\POp^{mod}}(r)$ can be seen as linear combinations of decorated trees with one marked vertex. The special vertex is decorated by an element of $\POp$. The other vertices are decorated by elements of $\bar \POp$.
The map $f$ is defined as follows. On a decorated tree $T\in \overline{\Bar\POp^{mod}}(r)$ we have $f(T)=f'(T)+f''(T)$ with $f'(T)\in \Q c\otimes \overline{\Bar\POp}$ and $f''(T)\in \MP'\otimes \overline{\Bar\POp}$, using the basis element $c$ and the sub-cyclic sequence $\MP'\subset \MP$ of the preceding subsection.
\begin{itemize}
\item $f'(T)$ is obtained by forgetting the marking at the marked vertex, projecting its decoration to $\bar \POp$.
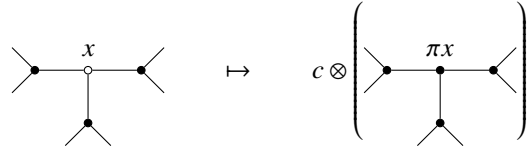
\[
\begin{tikzpicture}
    \node[ext,label=90:{$x$}] (v) at (0,0) {};
    \node[int] (w1) at (.7,0) {};
    \node[int] (w2) at (-.7,0) {};
    \node[int] (w3) at (0,-.7) {};
    \draw (v) edge (w1) edge (w2) edge (w3)
    (w1) edge +(.3,.3) edge +(.3,-.3)
    (w2) edge +(-.3,.3) edge +(-.3,-.3)
    (w3) edge +(.3,-.3) edge +(-.3,-.3)
    ;
\end{tikzpicture} 
\quad\quad
\mapsto 
\quad\quad
c\otimes \left(
\begin{tikzpicture}
    \node[int,label=90:{$\pi x$}] (v) at (0,0) {};
    \node[int] (w1) at (.7,0) {};
    \node[int] (w2) at (-.7,0) {};
    \node[int] (w3) at (0,-.7) {};
    \draw (v) edge (w1) edge (w2) edge (w3)
    (w1) edge +(.3,.3) edge +(.3,-.3)
    (w2) edge +(-.3,.3) edge +(-.3,-.3)
    (w3) edge +(.3,-.3) edge +(-.3,-.3)
    ;
\end{tikzpicture} 
\right)
\]

\item $f''(T)$ is zero unless the special vertex in $T$ is bivalent and connected to a leg (the $j$-th), whence we set $f''(T) = \pm\partial_j\otimes T' $, with $T'$ obtained by deleting the special vertex, applying the augmentation to its decoration.
\[
    \begin{tikzpicture}
        \node[ext,label=90:{$x$}] (v) at (0,0) {};
        \node[int] (w1) at (.7,0) {};
        \node  at (-.7,0) {$j$};
        \node (w2) at (1.2,0.1) {$\vdots$};
        \draw (v) edge +(-.5,0) edge (w1) 
        (w1) edge +(.3,.3) edge +(.3,-.3) edge +(.3,0)
        ;
    \end{tikzpicture} 
    \quad\quad
    \mapsto 
    \quad\quad
    \pm
    \epsilon(x) 
    \partial_j
    \otimes
    \left(
    \begin{tikzpicture}
        \node[int] (w1) at (.7,0) {};
        \node (w2) at (1.2,0.1) {$\vdots$};
        \node  at (-.5,0) {$j$};
        \draw 
        (w1) edge +(.3,.3) edge +(.3,-.3) edge +(.3,0) edge +(-1,0)
        ;
    \end{tikzpicture} 
    \right)
\]
To fix the sign, note that the internal edges formally carry degree $-1$ in the bar construction. The set-wise tensor product in \eqref{equ:bar def} means that implicitly we fix an ordering of the set of edges, and identify two orderings up to sign. The sign in the above formula is "+" if the edge at the marked vertex that is removed is the first in the ordering.

\end{itemize}

We claim that $f$ is a morphism of dg symmetric sequences. It is clear that $f$ is compatible with the symmetric group actions, but we have to check that it intertwines the differentials, i.e., 
\[
  D' f(T) = D f(T)  
\] 
for any decorated tree $T\in\overline{\Bar\POp^{mod}}$. The special vertex in such a tree $T$ is decorated by an element of $\POp=1\oplus \bar \POp$. 
We may assume that this special vertex in our generator $T$ is either decorated by $1$ or by $\bar \POp$. In either case we say that the \emph{essential vertices} in the tree are those that are decorated by $\bar P$. These are all non-marked vertices and the special vertex if decorated by $\bar P$.

Using this notation, let us split the differential on $\overline{\Bar\POp^{mod}}$ into terms 
\[
D = d_{\POp} + d_{c}' + d_c'',   
\]
where $d_{\POp}$ is induced by the internal differential on $\POp$, $d_{c}'$ contracts an edge between two essential vertices, and $d_c''$ contracts an edge between an essential vertex and the non-essential vertex (if the latter is present).

The differential $D'$ on $\MP\otimes \overline{\Bar\POp}$ has the form 
\[
    D'= d_{\POp} + d_c + \partial
\]
where $d_{\POp} + d_c$ stems from the differential on $\Bar\POp$ and $\partial$ is the differential on $\MP$, that is, $\partial \partial_i = c$.
We evidently have that 
\[
f(d_{\POp}T) = d_{\POp} f(T).
\]
Furthermore, since the projection $\pi$ to $\bar\POp$ commutes with the composition we have 
\[
    f(d_c'T) =  d_c f(t).  
\]
This leaves us with the verification that 
\begin{equation}\label{equ:MMPs dcpp tbs}
f(d_c''T) = \partial f(T).    
\end{equation}
Here we may suppose that the marked vertex of $T$ is decorated by $1$, otherwise both sides of \eqref{equ:MMPs dcpp tbs} are trivially 0. We then distinguish two cases:
\begin{itemize}
\item The marked vertex is adjacent to two other (essential) vertices. In this case $f(T)=0$. But we also have $f(d_c''T)=0$, since the contraction $d_c''T$ has two terms that are mapped to the same term in $\tMMP(\Bar\POp)$, but with opposite sign.
\[
    \begin{tikzpicture}
        \node[ext,label=90:{$1$}] (v) at (0,0) {};
        \node[int] (w1) at (-.7,0) {};
        \node[int] (w2) at (.7,0) {};
        \draw (v) edge (w1) edge (w2) 
        (w1) edge +(-.3,.3) edge +(-.3,-.3) edge +(-.3,0)
        (w2) edge +(.3,.3) edge +(.3,-.3) edge +(.3,0)
        ;
    \end{tikzpicture} 
    \quad 
    \xrightarrow{d_c''}
    \quad 
    \begin{tikzpicture}
        \node[ext] (w1) at (-.5,0) {};
        \node[int] (w2) at (.5,0) {};
        \draw (w1) edge (w2) 
        (w1) edge +(-.3,.3) edge +(-.3,-.3) edge +(-.3,0)
        (w2) edge +(.3,.3) edge +(.3,-.3) edge +(.3,0)
        ;
    \end{tikzpicture} 
    -
    \begin{tikzpicture}
        \node[int] (w1) at (-.5,0) {};
        \node[ext] (w2) at (.5,0) {};
        \draw (w1) edge (w2) 
        (w1) edge +(-.3,.3) edge +(-.3,-.3) edge +(-.3,0)
        (w2) edge +(.3,.3) edge +(.3,-.3) edge +(.3,0)
        ;
    \end{tikzpicture} 
    \quad 
    \xrightarrow{f}
    \quad 
    \pm
    c\otimes \left(
    \begin{tikzpicture}
        \node[int] (w1) at (-.5,0) {};
        \node[int] (w2) at (.5,0) {};
        \draw (w1) edge (w2) 
        (w1) edge +(-.3,.3) edge +(-.3,-.3) edge +(-.3,0)
        (w2) edge +(.3,.3) edge +(.3,-.3) edge +(.3,0)
        ;
    \end{tikzpicture}
    -
    \begin{tikzpicture}
        \node[int] (w1) at (-.5,0) {};
        \node[int] (w2) at (.5,0) {};
        \draw (w1) edge (w2) 
        (w1) edge +(-.3,.3) edge +(-.3,-.3) edge +(-.3,0)
        (w2) edge +(.3,.3) edge +(.3,-.3) edge +(.3,0)
        ;
    \end{tikzpicture}
    \right)
    =0
\]
\item The marked vertex is adjacent to one other vertex and a hair $j$. In this case $f(T)= \partial_j\otimes T'$ and $\partial f(T)=T'$.
$$
\begin{tikzpicture}
    \node[ext,label=90:{$1$}] (v) at (0,0) {};
    \node[int] (w1) at (.7,0) {};
    \node  at (-.7,0) {$j$};
    \node (w2) at (1.2,0.1) {$\vdots$};
    \draw (v) edge +(-.5,0) edge (w1) 
    (w1) edge +(.3,.3) edge +(.3,-.3) edge +(.3,0)
    ;
\end{tikzpicture} 
\quad\quad
\xrightarrow{f} 
\quad\quad
\partial_j
\otimes
\left(
\begin{tikzpicture}
    \node[int] (w1) at (.7,0) {};
    \node (w2) at (1.2,0.1) {$\vdots$};
    \node  at (-.5,0) {$j$};
    \draw 
    (w1) edge +(.3,.3) edge +(.3,-.3) edge +(.3,0) edge +(-1,0)
    ;
\end{tikzpicture} 
\right)
\xrightarrow{\partial}
c
\otimes
\left(
\begin{tikzpicture}
    \node[int] (w1) at (.7,0) {};
    \node (w2) at (1.2,0.1) {$\vdots$};
    \node  at (-.5,0) {$j$};
    \draw 
    (w1) edge +(.3,.3) edge +(.3,-.3) edge +(.3,0) edge +(-1,0)
    ;
\end{tikzpicture} 
\right)
$$
To fix the sign, we assume that the edge incident at the special vertex is the first in the ordering.
On the other hand, $d_c''T=T'$ has only one term, hence we conclude that \eqref{equ:MMPs dcpp tbs} holds in this case. 
$$
\begin{tikzpicture}
    \node[ext,label=90:{$1$}] (v) at (0,0) {};
    \node[int] (w1) at (.7,0) {};
    \node  at (-.7,0) {$j$};
    \node (w2) at (1.2,0.1) {$\vdots$};
    \draw (v) edge +(-.5,0) edge (w1) 
    (w1) edge +(.3,.3) edge +(.3,-.3) edge +(.3,0)
    ;
\end{tikzpicture} 
\quad\quad
\xrightarrow{d_c''} 
\quad\quad
\begin{tikzpicture}
    \node[ext] (w1) at (.7,0) {};
    \node (w2) at (1.2,0.1) {$\vdots$};
    \node  at (-.5,0) {$j$};
    \draw 
    (w1) edge +(.3,.3) edge +(.3,-.3) edge +(.3,0) edge +(-1,0)
    ;
\end{tikzpicture} 
\xrightarrow{f}
c
\otimes
\left(
\begin{tikzpicture}
    \node[int] (w1) at (.7,0) {};
    \node (w2) at (1.2,0.1) {$\vdots$};
    \node  at (-.5,0) {$j$};
    \draw 
    (w1) edge +(.3,.3) edge +(.3,-.3) edge +(.3,0) edge +(-1,0)
    ;
\end{tikzpicture} 
\right)
$$
\end{itemize}
Note that the case that the special vertex is connected to two legs does not need to be considered: This applies to only the graph 
\begin{equation*}
    \begin{tikzpicture}
        \node[ext,label=90:{$1$}] (v) at (0,0) {};
        \node (e1) at (-.7,0) {$1$};
        \node (e2) at (.7,0) {$2$};
        \draw (v) edge (e1) edge (e2)
        ;
    \end{tikzpicture}
     ,
\end{equation*}
which is however removed by passing to the coaugmentation coideal.

We next check that $f$ is a quasi-isomorphism.
To this end we introduce ascending exhaustive filtrations on $\overline{\Bar\POp^{mod}}$ and $\MP\otimes \overline{\Bar\POp}$ by the number of essential vertices in trees. Hence let $\mF^p \overline{\Bar\POp^{mod}}\subset\overline{ \Bar\POp^{mod}}$ (resp. $\mF^p(\MP\otimes \overline{\Bar \POp})\subset \MP\otimes \overline{\Bar \POp}$ ) the sub-symmetric sequences spanned by trees with $\leq p$ essential vertices.
The map $f$ is compatible with these filtrations, 
$$
f(\mF^p \overline{\Bar\POp^{mod}})\subset \mF^p(\MP\otimes \overline{\Bar \POp}).
$$
Hence it suffices to show that the associated graded morphism $$
\gr_{\mF} f : \gr_{\mF}\overline{\Bar\POp^{mod}} 
\cong 
(\overline{\Bar\POp^{mod}}, d_{\POp} + d_c'')
\to 
\gr_{\mF} (\MP\otimes \overline{\Bar \POp})
\cong 
(\MP\otimes \overline{\Bar \POp}, d_{\POp}+\partial)
$$
is a quasi-isomorphism. (Then $f$ is a quasi-isomorphism as well.) Filtering further by minus the total degree of the $\bar P$-decorations, it a fortiori suffices to show that the morphism 
$$
\gr_{\mF} f : 
(\overline{\Bar\POp^{mod}}, d_c'')
\to 
(\MP\otimes \overline{\Bar \POp}, \partial)
$$
is a quasi-isomorphism.
The cohomology of the arity $r$ part of the right-hand side is $\MMP\otimes \overline{\Bar \POp}$.
To compute the cohomology of the left-hand side, we consider a decomposition
\[
    \overline{\Bar\POp^{mod}}((r)) = U\oplus V
\]
with $V\subset \overline{\Bar\POp^{mod}}(r)$ spanned by decorated trees all of whose vertices are decorated by $\bar \POp$, and $U\subset \overline{\Bar\POp^{mod}}(r)$ is spanned by trees with the marked vertex decorated by $1$.
Then $d_c'':U\to V$ maps $U$ to $V$. It is surjective, hence the cohomology is equal to the kernel of $d_c''$.
But by inspection this is isomorphic to $\MMP\otimes \overline{\Bar \POp}((r))$ (as desired), with the generator $(e_1-e_j)\otimes T$ represented by the linear combination of trees 
\[(e_1-e_j)\otimes 
   \left( \begin{tikzpicture}[scale=.7]
        \node (e1) at (0,0) {$\scriptstyle 1$};
        \node (e2) at (2,0) {$\scriptstyle j$};
        \node[int] (v1) at (0.5,0) {};
        \node at (1,.6) {$\cdots$};
        \node at (1,-.6) {$\cdots$};
    \node[int] (v2) at (1,0) {};
    \node[int] (v3) at (1.5,0) {};
    \draw
    (v1)  edge +(0,.3) edge +(0,-.3) edge (e1) 
    (v2) edge (v1) edge (v3)  edge +(0,.3) edge +(0,-.3)
    (v3) edge (e2) edge +(0,.3) edge +(0,-.3);
    \end{tikzpicture}
    \right)
\sim
\begin{tikzpicture}[scale=.7]
    \node (e1) at (-1,0) {$\scriptstyle 1$};
    \node (e2) at (3,0) {$\scriptstyle j$};
    \node[int] (v1) at (0,0) {};
    \node at (1,.6) {$\cdots$};
    \node at (1,-.6) {$\cdots$};
\node[int] (v2) at (1,0) {};
\node[int] (v3) at (2,0) {};
\node[ext,label=90:{$\scriptstyle 1$}] (w) at (-.5,0) {};
\draw (w) edge (e1) edge (v1) 
(v1)  edge +(0,.3) edge +(0,-.3)
(v2) edge (v1) edge (v3)  edge +(0,.3) edge +(0,-.3)
(v3) edge (e2) edge +(0,.3) edge +(0,-.3);
\end{tikzpicture}
+ 
\begin{tikzpicture}[scale=.7]
    \node (e1) at (-1,0) {$\scriptstyle 1$};
    \node (e2) at (3,0) {$\scriptstyle j$};
    \node[int] (v1) at (0,0) {};
    \node at (1,.6) {$\cdots$};
    \node at (1,-.6) {$\cdots$};
\node[int] (v2) at (1,0) {};
\node[int] (v3) at (2,0) {};
\node[ext,label=90:{$\scriptstyle 1$}] (w) at (.5,0) {};
\draw (w)  edge (v1) edge (v2)
(v1)  edge +(0,.3) edge +(0,-.3) edge (e1)
(v2)  edge (v3)  edge +(0,.3) edge +(0,-.3)
(v3) edge (e2) edge +(0,.3) edge +(0,-.3);
\end{tikzpicture}
+
\begin{tikzpicture}[scale=.7]
    \node (e1) at (-1,0) {$\scriptstyle 1$};
    \node (e2) at (3,0) {$\scriptstyle j$};
    \node[int] (v1) at (0,0) {};
    \node at (1,.6) {$\cdots$};
    \node at (1,-.6) {$\cdots$};
\node[int] (v2) at (1,0) {};
\node[int] (v3) at (2,0) {};
\node[ext,label=90:{$\scriptstyle 1$}] (w) at (1.5,0) {};
\draw (w)  edge (v2) edge (v3) 
(v1)  edge +(0,.3) edge +(0,-.3) edge (e1)
(v2) edge (v1) edge +(0,.3) edge +(0,-.3)
(v3) edge (e2) edge +(0,.3) edge +(0,-.3);
\end{tikzpicture}
+
\begin{tikzpicture}[scale=.7]
    \node (e1) at (-1,0) {$\scriptstyle 1$};
    \node (e2) at (3,0) {$\scriptstyle j$};
    \node[int] (v1) at (0,0) {};
    \node at (1,.6) {$\cdots$};
    \node at (1,-.6) {$\cdots$};
\node[int] (v2) at (1,0) {};
\node[int] (v3) at (2,0) {};
\node[ext,label=90:{$\scriptstyle 1$}] (w) at (2.5,0) {};
\draw (w) edge (e2) edge (v3) 
(v1)  edge +(0,.3) edge +(0,-.3) edge (e1)
(v2) edge (v1) edge (v3)  edge +(0,.3) edge +(0,-.3)
(v3) edge +(0,.3) edge +(0,-.3);
\end{tikzpicture}
\] 
obtained by summing over ways of introducing a bivalent special vertex on an edge along the path from hair $1$ to hair $j$ in the tree. It is hence clear that the map $\gr_{\mF}f$ induces an isomorphism on cohomology as desired.
\end{proof}

\begin{rem}\label{rem:MMPs graded}
We will use below that the above proof in fact shows a slightly stronger statement than claimed in the proposition: $f$ induces a quasi-isomorphism already on the level of the associated graded objects with respect to the bounded below exhaustive filtration by the number of essential vertices.
\end{rem}





\begin{rem}
The zigzag of Proposition \ref{prop:MMPs} is in fact a zigzag of strict quasi-isomorphisms of $\infty$-$\Bar\POp$-comodules for natural $\infty$-$\Bar\POp$-comodule on the objects. 
This will however not be needed for the present paper.
\end{rem}

\section{The adjoint functor $G$}

\subsection{The forgetful functor $F$}
The functor 
\[
F : \dg\cycOp \to \dgptPairOp    
\]
sending an augmented cyclic operad $\POp$ to the pair $(\POp^{nc},\POp^{mod})$ has been introduced in the introduction.
We understand the pair as a pointed object with the morphisms
\[
  \peis\to  (\POp^{nc},\POp^{mod}) \to \peis
\]
defined via the unit and augmentation of $\POp$ respectively.
We also emphasize that the right-module $\POp^{mod}$ has underlying symmetric sequence 
\[
\POp^{mod}(r) = \POp((r)).
\]

\subsection{Construction of the left adjoint $G$}\label{sec:G constr}
We construct the adjoint functor $G$ of Theorem \ref{thm:main dg} explicitly. Let $\peis\to (\QOp, \MOp)\to \peis$ be an object of $\dgptPairOp$. Then we define 
\[
G(\QOp, \MOp) = \Free_{cyc}(\Ind \bar{\QOp}\oplus \MOp)/\sim ,
\]
with the following relations:
\begin{itemize}
    \item For $q_1, q_2\in \bar \QOp$ we again denote by $\tilde q_j$ the corresponding elements in $\Ind \bar{\QOp}$. Then we have:
    \begin{equation}\label{equ:Grel 1}
     \tilde q_1 \circ_{j,0} \tilde q_2 =
      \widetilde{(q_1\circ_j q_2)}.
    \end{equation}
    \item For $q\in \bar \QOp$ and $m\in \MOp$:
    \begin{equation}\label{equ:Grel 2}
    m \circ_{i,0} \tilde q 
    =
    m\circ_i q.    
    \end{equation}
    \item Let $m_1\in \MOp(2)$ be the image of the generator of $\meis$. Then 
    \begin{equation}\label{equ:Grel 3}
    m_1 = 1    
    \end{equation}
    is identified with the operadic unit of $G(\QOp, \MOp)$.
\end{itemize}
It is clear that the ideal in the free cyclic operad generated by the relations is a dg ideal.
We also have that the cyclic operad $G(\QOp, \MOp)$ inherits the augmentation from the pair $(\QOp, \MOp)$.
It is also clear that $G:\dgptPairOp\to \dg\cycOp$ defines a functor.

\begin{lemma}
The functor $G$ is left Quillen adjoint to the functor $F$.
\end{lemma}
\begin{proof}
    We verify the adjunction relation explicitly.
    Consider a morphism of augmented cyclic operads
\[
\Phi: G(\QOp, \MOp) \to \POp.   
\]
$\Phi$ is fully determined by its restriction to generators 
\begin{align*}
f: \Ind\bar\QOp &\to \POp 
&
g: \MOp &\to \POp.
\end{align*}
Here $f$ and $g$ are morphisms of cyclic sequences.
Compatibility with the augmentations holds iff $f$ takes values in $\bar \POp$ and $g\mid_{\bar \MOp}$ takes values in $\bar \POp$.

By the induction/restriction adjunction we then have that the data $f$ is equivalent to a morphism of (non-cyclic) symmetric sequences
\[
 \bar f: \bar\QOp \to \Res \bar \POp.
\]
Furthermore, the relations \eqref{equ:Grel 1} are respected (i.e., annihilated) by $\Phi$ iff $\bar f$ is a morphism of non-unital operads.

Similarly, the relations \eqref{equ:Grel 2} are respected by $\Phi$ iff the map $g$ is a morphism of operadic right modules, and relation \eqref{equ:Grel 3} is satisfied iff that map of right modules is a morphism of the under-category under $\meis$.

Summarizing, cyclic operad morphisms $\Phi$ as above are the same data as pairs consisting of a morphism of augmented non-cyclic operads $\QOp\to \POp^{nc}$ and a pointed right module morphism $\MOp\to \POp^{mod}$, as was to be shown.
\end{proof}

\subsubsection{Alternative definition of $G$}\label{sec:G constr alt}
Let us also note that the functor $G$ can be constructed more abstractly as follows.
First, we consider $G$ on free objects 
\[
    \Free_{\pt pair}(\AOp, \BOp)
    =
    (\Free \AOp, \Free_{\ptMod_{\Free A}} \BOp), 
\]
with $\AOp$ and $\BOp$ dg symmetric sequences, and $\Free_{\ptMod_{\Free A}} \BOp$ the free pointed module generated by $\BOp$. This is the free module generated by $\BOp\oplus\meis$.
We may then set 
\[
 G ( \Free_{\pt pair}(\AOp, \BOp)  )
 =
 \Free_{cyc}((\Ind\AOp) \oplus\BOp)
\]
to be the free cyclic operad generated by $\Ind\AOp$ and $\BOp$.
The construction $G(-)$ extends naturally to the full subcategory of free objects in $\dg\pt\PairOp$.


Let $(\QOp, \MOp)$ be a(nother) pair consisting of a non-cyclic operad $\QOp$ and a pointed right $\QOp$-module $\MOp$.
We can always write $(\QOp, \MOp)$ as a coequalizer of free objects
\[
    (\QOp, \MOp)
    \cong
    \coeql \left(\Free_{\pt pair}\overline{\Free_{\pt pair}(\bar \QOp, \bar \MOp)}
    \rightrightarrows
    \Free_{\pt pair}( \bar \QOp,  \bar \MOp) \right).
\]
We then define 
\begin{equation}\label{equ:G def alt}
    \begin{aligned}
    G(\QOp, \MOp) &= 
    \coeql \left( G(\Free_{\pt pair}\overline{\Free_{\pt pair}(\bar \QOp, \bar \MOp)}) 
    \rightrightarrows
    G(\Free_{\pt pair}( \bar \QOp,  \bar \MOp)) \right)
    \\&=
    \coeql \left( \Free_{cyc}(\Ind\overline{\Free_{\Op}\bar \QOp}\oplus \overline{\Free_{\pt\Mod-\Free_{\Op}\bar \QOp} \bar \MOp})
    \rightrightarrows
    \Free_{cyc}( \Ind \bar\QOp\oplus  \bar \MOp) \right)
    .
\end{aligned}
\end{equation}

\subsubsection{Graphical description of $G(\QOp, \MOp)$ }
In particular $G(\QOp, \MOp)$ is a quotient of the free cyclic operad generated by $\Ind \bar \QOp$ and $\MOp$.
Elements of this free cyclic operad can be seen as linear combinations of decorated unrooted trees, with decorations of vertices in $\Ind \bar \QOp$ or $\MOp$.
A decoration by $\Ind \bar \QOp$ can be understood as a $\bar \QOp$-decoration plus a marking of one half-edge incident at the vertex, indicating the "output" of the $\bar \QOp$-decoration.
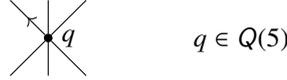
\[
\begin{tikzpicture}
    \node[int, label=0:{$q$}] (v) at (0,0) {};
    \draw (v) edge +(.5,.5) edge +(0,.5) edge[->-] +(-.5,.5) edge +(.5,-.5) edge +(0,-.5) edge +(-.5,-.5);
\end{tikzpicture}
\quad \quad\quad\quad q\in \QOp(5)
\]

The relations imposed by the coequalizer are the following relations:
\begin{itemize}
\item If decorations on neighboring vertices in a tree can be composed, then we equate the tree with the one obtained by composing the decorations and contracting the connecting edge:
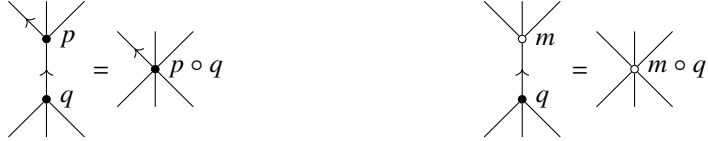
\begin{align*}
    \begin{tikzpicture}
        \node[int, label=0:{$p$}] (v) at (0,.4) {};
        \node[int, label=0:{$q$}] (w) at (0,-.4) {};
        \draw (w) edge[->-] (v) edge +(.5,-.5) edge +(0,-.5) edge +(-.5,-.5)
        (v) edge +(.5,.5) edge +(0,.5) edge[->-] +(-.5,.5);
    \end{tikzpicture}  
    &=
    \begin{tikzpicture}
        \node[int, label=0:{$p\circ q$}] (v) at (0,0) {};
        \draw (v) edge +(.5,.5) edge +(0,.5) edge[->-] +(-.5,.5) edge +(.5,-.5) edge +(0,-.5) edge +(-.5,-.5);
    \end{tikzpicture}
    &
    \begin{tikzpicture}
        \node[ext, label=0:{$m$}] (v) at (0,.4) {};
        \node[int, label=0:{$q$}] (w) at (0,-.4) {};
        \draw (w) edge[->-] (v) edge +(.5,-.5) edge +(0,-.5) edge +(-.5,-.5)
        (v) edge +(.5,.5) edge +(0,.5) edge +(-.5,.5);
    \end{tikzpicture}  
    &=
    \begin{tikzpicture}
        \node[ext, label=0:{$m\circ q$}] (v) at (0,0) {};
        \draw (v) edge +(.5,.5) edge +(0,.5) edge +(-.5,.5) edge +(.5,-.5) edge +(0,-.5) edge +(-.5,-.5);
    \end{tikzpicture}
\end{align*}
Here we marked the $\QOp$-decorated vertices by a black dot and the $\MOp$-decorated vertices by a white dot.
\item The operadic unit in $\QOp$ is identified with the unit in the free cyclic operad
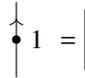
\[
    \begin{tikzpicture}
        \node[int, label=0:{$1$}] (v) at (0,0) {};
        \draw (v) edge[->-] +(0,.5) edge +(0,-.5);
    \end{tikzpicture}   
    =
    \begin{tikzpicture}
        \draw (0,-.4)--(0,.4);
    \end{tikzpicture}  
\]
\item The marked element $1\in \MOp((2))$ is also identified with the operadic unit in the free cyclic operad.
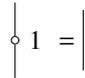
\[
    \begin{tikzpicture}
        \node[ext, label=0:{$1$}] (v) at (0,0) {};
        \draw (v) edge +(0,.5) edge +(0,-.5);
    \end{tikzpicture}   
    =
    \begin{tikzpicture}
        \draw (0,-.4)--(0,.4);
    \end{tikzpicture}  
\]
\end{itemize}

\subsection{$G$ of a cobar construction}
Our next goal is to understand the value of the adjoint functor $G$ 
\[
G(\Bar^c(\COp, \MOp))    
\]
on the cobar construction 
\[
    \Bar^c(\COp, \MOp)
    =
    (\Bar^c\COp, \Bar^c_{\COp}\MOp)   
\]
of a pair $(\COp, \MOp)$ consisting of a 1-shifted cooperad $\COp$ and a pointed $\COp$-comodule $\MOp$.
First, by quasi-freeness of the cobar construction we have that 
\[
G(\Bar^c(\COp, \MOp))  
\cong 
(\Free_{cyc}(\Ind\bar \COp \oplus \bar \MOp), D) 
\]
is a quasi-free cyclic operad generated by $\Ind\bar \COp$ and $\bar \MOp$, with differential yet to be determined.
We then have:
\begin{lemma}
There is a 1-shifted non-unital cyclic dg cooperad structure on the cyclic sequence $\Ind\bar \COp \oplus \bar \MOp$ such that $G(\Bar^c(\COp, \MOp))$ is identified with the cobar construction of that 1-shifted non-unital cyclic cooperad,
\[
    G(\Bar^c(\COp, \MOp)) \cong \Bar^c(\Ind\bar \COp \oplus \bar \MOp).  
\]
\end{lemma}
In fact, this Lemma is almost a tautology: A1-shifted cyclic $\infty$-cooperad structure on is determined by a codifferential on $\Ind\bar \COp \oplus \bar \MOp$ is by definition a degree +1 square zero derivation on the augmented cyclic operad $\Free_{cyc}(\Ind\bar \COp \oplus \bar \MOp)$, and we are given one such derivation in the form of the differential $D$. So one just has to check that the $\infty$-cooperad structure defined by $D$ has no operations of arity $\geq 3$. But this is also clear since the differential $D$ is induced solely by the differentials on $\Bar^c(\COp, \MOp)$, which have no terms that would make a generator into more than two factors.

Nevertheless, let us make the dg cooperad structure defined by $D$ on $\Ind\bar \COp \oplus \bar \MOp$ explicit. The differential on the generators $\Ind\bar \COp$ is defined (via monoidal functoriality of $\Ind$) by the differential $D'$ on the generators $\bar \COp$ of the cobar construction $\Bar^c\COp$.
Concretely, for $c\in \bar\COp$ this is given by 
\[
D' c = d_{\bar\COp} + \sum c' \circ c''    
\]
where the first term is the internal differential on $\bar\COp$, and the second is Sweedler notation for the reduced cocomposition, followed by composition in $\Bar^c\COp$.
The analogous formula also holds in $G(\Bar^c(\COp, \MOp))$, one just applies $\Ind$.
It follows that our desired non-unital 1-shifted dg cyclic cooperad structure on $\Ind\bar \COp \oplus \bar \MOp$ is such that $\Ind\bar \COp$ is a sub-cooperad, with the cooperad structure the same as defined via the $(\Ind,\Res)$-adjunction between cooperads and cyclic cooperads, see section \ref{sec:op cycop adj}.

We next turn to the generators $\bar \MOp$.
Again the differential $D$ on $\bar \MOp$ is determined by the differential $D'$ on the comodule cobar construction $\Bar^c_{\COp}\MOp$. For $m\in \bar \MOp$ we have
\[
D'm = d_{\MOp} m + \sum m'\circ c'', 
\]
with the first term the internal differential on $\MOp$ and the second being Sweedler notation for the $\bar \COp$-coaction, followed by the composition in $\Bar^c_{\COp}\MOp$.
Note however that $\Bar^c_{\COp}\MOp$ is generated by $\MOp$, not $\bar \MOp$, and the second term above may have terms $m'\notin \bar \MOp$.
Separating these terms we may write 
\[
    D'm = d_{\MOp} m + 
    \sum \epsilon(m') m_1 \circ c''
    +
    \sum \pi m' \circ c'',   
\]
with $m_1\in \MOp(2)$ the image of the generator under $\meis\to \MOp$, $\epsilon:\MOp\to \meis$, and $\pi m' = m'-m_1\epsilon(m')$ the projection to $\bar \MOp$.
Note that by definition of $G$ in the previous subsection the element $m_1$ is identified with the unit in the cyclic operad $G(\Bar^c(\COp, \MOp))$.
Hence the differential $D$ there is given on $m\in \bar \MOp$ by 
\[
D m = d_{\MOp} m + d'm + \sum \pi m' \circ c''
\]
with $d':\MOp \to \bar \COp$ a new part of the differential induced by the map $m\mapsto \sum \epsilon(m') c''$.

\subsection{Proof of Theorem \ref{thm:main dg}}

Let $\POp$ be an augmented cyclic operad.
It is is automatically fibrant since all objects in $\dg\cycOp$ are fibrant. 
The goal is to show that the derived counit of the adjunction 
\[
LG (F(\POp))  \to \POp 
\]
is a quasi-isomorphism.
To this end we choose as a cofibrant replacement of $F(\POp)=(\POp^{nc}, \POp^{mod})$ in $\dgptPairOp$ the bar-cobar resolution.
We hence have to check that the natural morphism 
\[
\eta \colon G\Bar^c\Bar F (\POp) \to \POp
\]
induced by the adjunction counits $\Bar^c\Bar F(\POp)\to F(\POp)$ and $G\circ F \Rightarrow \mathit{id}$
is a quasi-isomorphism of cyclic operads.
Let us first make the morphism $\eta$ more explicit.
By the discussion of the previous subsection we have 
\[
G\Bar^c\Bar F (\POp)
=
\Bar_{cyc}^c( \Ind\overline{\Bar\POp^{nc}} \oplus 
\overline{\Bar \POp^{mod}} ).
\]  
The morphism $\eta$ is fully determined by its value on generators. Tracing the definition of the counits, the restriction of $\eta$ to the generators $\Ind \overline{\Bar\POp^{nc}}$ is given by the map 
\[
    \Ind\overline{\Bar\POp^{nc}}
    \to 
    \Ind \overline{\POp^{nc}} \to \bar \POp
\]
that is the composition of the projection to the cogenerators of the (non-cyclic operadic) bar construction followed by the counit of the $(\Ind, \Res)$-adjunction of section \ref{sec:symseq}, i.e., the forgetful map.
The restriction of $\eta$ to the generators $\overline{\Bar \POp^{mod}}$ is similarly given by the projection 
\[
    \overline{\Bar \POp^{mod}} \to \bar\POp    
\]
to the cogenerators of the module cobar construction.
We next claim that one can factorize $\eta$ as a composition passing through the cyclic operadic bar-cobar resolution of $\POp$
\[
\Bar^c\Bar F (\POp)
=
\Bar_{cyc}^c( \Ind\overline{\Bar\POp^{nc}} \oplus 
\overline{\Bar \POp^{mod}} )
\xrightarrow{\Bar_{cyc}^c\phi}
\Bar_{cyc}^c(\Bar_{cyc}\POp) \xrightarrow{\eta_{cyc}} \POp,
\]
with the right-hand morphism $\eta_{cyc}$ the counit of the cyclic operad bar-cobar adjunction and
\[
\phi :  (\Ind\overline{\Bar\POp^{nc}} \oplus 
\overline{\Bar \POp^{mod}}, D)
\to 
\overline{\Bar_{cyc}\POp}  
\]
the following morphism of non-unital 1-shifted cyclic cooperads:
Since $\overline{\Bar_{cyc}\POp}$ is quasi-cofree the morphism $\phi$ is uniquely determined by its composition $\pi\circ \phi$ with the projection $\pi: \overline{\Bar_{cyc}\POp} \to \bar \POp$ to cogenerators.
This composition $\pi\circ \phi$ on $\Ind\overline{\Bar\POp^{nc}}$ is the composition of the projection to cogenerators with the counit under the $(\Ind, \Res)$-adjunction 
\[
    \Ind\overline{\Bar\POp^{nc}}
    \to 
    \Ind\bar\POp^{nc}
    \to 
    \bar\POp.
\]
On the summand $\overline{\Bar \POp^{mod}}$ the morphism $\pi\circ \phi$ is the composition of the projection to cogenerators of the module bar construction with the projection to the augmentation ideal
\[
    \overline{\Bar \POp^{mod}}
\to \POp
\xrightarrow{\pi}
\bar \POp.
\]
We leave it to the reader to check that $\phi$ indeed respects the differentials, and is hence a morphism of non-unital 1-shifted cyclic dg cooperads.
It is clear by construction that $\eta=(\Bar_{cyc}^c\phi)\circ \eta_{cyc}$.
Furthermore, $\eta_{cyc}$ is a quasi-isomorphism, so we are left with checking that $\Bar_{cyc}^c\phi$ is a quasi-isomorphism as well.
To do this we will check that $\phi$ is a quasi-isomorphism on the associated graded level and then invoke Lemma \ref{lem:cyclic cobar} saying that the cobar construction sends such quasi-isomorphisms to quasi-isomorphisms.

Recall also Proposition \ref{prop:MMPs}. We shall construct a commutative diagram of dg cyclic sequences 
\begin{equation}\label{equ:main proof diag}
\begin{tikzcd}
    (\Ind\overline{\Bar\POp^{nc}} \oplus 
    \overline{\Bar \POp^{mod}}, D)
    \ar{rr}{\phi}
    \ar{dr}{l}
    &
    &
    \overline{\Bar \POp} 
     \\
&
(\Ind\overline{\Bar\POp^{nc}} \oplus \MP\otimes \overline{\Bar \POp}, D)
=
(\Ind\overline{\Bar\POp} \oplus \Ind\overline{\Bar\POp}[1]\oplus \overline{\Bar\POp}
, D)
\ar{ur}{r}
&
\end{tikzcd}.
\end{equation}
We explain the second line and the arrows $
r,l$.
The construction $\MP\otimes \overline{\Bar \POp}$ is as in Proposition \ref{prop:MMPs}.
The dg cyclic sequence in the second line of \eqref{equ:main proof diag} is the mapping cone of the projection $\MP\otimes \overline{\Bar \POp}\to \MP'\otimes \overline{\Bar \POp} \cong\Ind\overline{\Bar\POp}[1]$, see section \ref{sec:auxseq} for the notation.
The arrow $l$ is just the prolongation of the morphism of Proposition \ref{prop:MMPs} by the identity map on $\Ind\overline{\Bar\POp}$.

The arrow $r$ is the following map:
The summand $\Ind\overline{\Bar\POp}[1]$ is sent to zero by $r$, on the summand $\overline{\Bar\POp}$ the map $r$ is the identity and on $\Ind\overline{\Bar\POp}$ the map $r$ is defined via the counit of the $(\Ind, \Res)$-adjunction. It is clear that $\phi=r\circ l$.
Note that by a simple spectral sequence argument
one has that the map $r$ is a quasi-isomorphism. 
Furthermore, $l$ is also a quasi-isomorphism, using Proposition \ref{prop:MMPs}.
Both these statements hold also on the level of the associated graded complexes with respect to the filtration by number of essential vertices, see Remark \ref{rem:MMPs graded}.
Hence we conclude that $\gr_{\mF}\phi$ is also a quasi-isomorphism.
Hence by Lemma \ref{lem:cyclic cobar} we conclude that $\Bar^c\phi$ is a quasi-isomorphism as desired.
\hfill\qed

\subsection{Version for dg cooperads}
The construction of the right adjoint functor $G$ in the dg cooperad setting is dual but otherwise identical to the construction of the right adjoint $G$ in the dg operad setting in section \ref{sec:G constr}. 
The existence of the right-adjoint $G$ to $F$ follows from standard adjoint functor theorems: The categories are locally presentably, and the forgetful functor $F$ preserves colimits, which are created in dg symmetric sequences for both categories involved,
We just note that there is a formula for $G$ analogous (dual) to \eqref{equ:G def alt}. 
Let $\peis^*\to (\COp,\MOp)\to \peis^*$ be an object of $\dgptPairOp^c$, consisting of a cougmented conilpotent dg cooperad $\COp$, and a $\COp$-comodule $\MOp$. We may then write $(\COp,\MOp)$ as an equalizer of cofree objects
\[
    (\COp, \MOp)
    \cong
    \eql \left(
        \Free_{\pt pair}^c( \bar \COp,  \bar \MOp)
        \rightrightarrows
        \Free_{\pt pair}^c\overline{\Free_{\pt pair}^c(\bar \COp, \bar \MOp)}
     \right).
\]
On cofree objects $\Free_{\pt pair}^c(\AOp,  \BOp)$ one quickly sees from the adjunction relation that 
\[
    G(\Free_{\pt pair}^c(\AOp,  \BOp)) = 
    \Free_{cyc}^c(\Ind\AOp\oplus\BOp).
\]
Then, since $G$ preserves colimits as a right adjoint $G(\COp, \MOp)$ is identified with the equalizer 
 \begin{equation}\label{equ:G def alt dual}
    G(\QOp, \MOp) 
    =
    \eql \left(
        \Free_{cyc}^c( \Ind \bar\COp\oplus  \bar \MOp)
        \rightrightarrows
        \Free_{cyc}^c\left(\Ind\overline{\Free_{\Op}^c\bar \COp}\oplus \overline{\Free_{\pt\Mod^c-\Free_{\Op}^c\bar \COp}^c \bar \MOp}\right)
     \right)
    .
\end{equation}

Furthermore, by dualizing the proof of the preceding subsection we obtain the following dg cooperad version of Theorem \ref{thm:main dg}.

\begin{thm}\label{thm:main dgc}
    The forgetful functor $F$ is part of a Quillen adjunction 
    \[
    F \colon \dg\cycOp^c  \rightleftarrows  \dg\pt\PairOp^c  \colon G.   
    \]
    For $\COp$ any coaugmented cyclic cooperad the derived unit of the adjunction 
    \[
    \COp \to RG( LF (\COp))    
    \] 
    is a weak equivalence.
\end{thm}
\hfill\qed

\subsection{Hopf cooperads and proof of Theorem \ref{thm:main Hopf}}
The construction of the adjoint $G$ in the dg cooperad setting of the previous section extends to the Hopf cooperadic setting.
The existence of the left-adjoint follows again from adjoint functor Theorems: The categories are locally presentable, and the functor $F$ preserves colimits, which are created arity-wise in dg commutative algebras.
Then, on cofree dg Hopf pairs $(\COp, \MOp)=\Free^c_{\pt\PairOp^c}(\AOp, \BOp)$ cogenerated by symmetric sequences in dg commutative algebras $\AOp$, $\BOp$, one sees from the adjunction relation that $G$ is the cofree cyclic cooperad cogenerated by $\AOp$ and $\BOp$,
\[
G(\Free^c(\AOp, \BOp)) = \Free_{cyc}^c(\Ind^c\AOp\oplus \BOp). 
\]
Note that the free cyclic operad inherits a natural Hopf structure, i.e., arity-wise commutative products, by tree-wise multiplication of elements in $\AOp$ and $\BOp$.
Finally, by expressing any pointed dg Hopf pair as an equalizer of cofree objects as before, we see that the equalizer \eqref{equ:G def alt dual} also defines the right adjoint $G$ on dg Hopf pairs.
But note that limits in the category $\cycHOpc$ are creted in $\dg\cycOp^c$, and hence the functor $G(\COp, \MOp)$ for Hopf pairs is the same dg cyclic cooperad as the one computed for the underlying dg pairs.

Hence the quasi-isomorphism statement in Theorem \ref{thm:main Hopf} readily follows from its non-Hopf version, namely Theorem \ref{thm:main dgc}.
More precisely, for $\COp$ a cofibrant cyclic dg Hopf cooperad, let $X$ be a fibrant replacement in $\pt\PairHOpc$ of the pair $F(\COp)=(\COp^{nc},\COp^{mod})$. Then $X$ is also fibrant in $\dg\PairOp^c$ by the Quillen adjunction \eqref{equ:hopf forget adj}. 
Hence the morphism $\COp\to G(X)$ is also the derived counit in the category $\cycHOpc$, and hence a weak equivalence by Theorem \ref{thm:main dgc}.
We also note that is is actually not required that $\COp$ is cofibrant, since $F$ preserves weak equivalences.
\hfill\qed

\section{Application: Homotopy automorphisms of $\BV^c$}

\subsection{Parenthesized ribbon chord diagrams}
We recall here known facts about chord diagrams.
The framed Drinfeld-Kohno operad in Lie algebras $\frt$ is defined such that
$\frt(r)$ is generated by symbols $t_{ij}$, $1\leq i,j\leq r$, with relations 
\begin{equation}\label{equ:drinfeldkohno_rel}
	\begin{aligned}
	t_{ij} &= t_{ji} \\
	[t_{ij}, t_{kl}] &= 0 \quad \text{for $\{i,j\}\cap\{k,l\}=\emptyset$ } \\
	[t_{ij},t_{ki}+t_{kj}] &= 0\, .
	\end{aligned}
\end{equation}
The Lie algebras $\frt(r)$ assemble into an operad in Lie algebras, see \cite[1.3]{Severa2010} for explicit formulas for the composition morphisms (with notation $s_j=\frac 1 2 t_{jj}$). The operad in Lie algebras $\frt$ also has a cyclic structure. Concretely, the cyclic structure map corresponding to the transposition $0\leftrightarrow 1$ is given by
\begin{align*}
t_{ij} & \mapsto t_{ij}  &
t_{1i} & \mapsto -\sum_{k=1}^r t_{ki} &
t_{11} & \mapsto \sum_{k,l=1}^r t_{kl}
\end{align*}
for $i,j\geq 2$.

Alternatively, $\frt(r)$ also has a manifestly cyclic presentation.
We add the generators $t_{00}$ and $t_{0i}=t_{i0}$, $i=1,\dots,r$. Then the relations can be re-written in a cyclically invariant form as 
\begin{align*}
    t_{ij} &= t_{ji} \\
	[t_{ij}, t_{kl}] &= 0 \quad \text{for $\{i,j\}\cap\{k,l\}=\emptyset$ } \\
	\sum_{i=0}^r t_{ij} &= 0\, .
\end{align*}
Yet alternatively, we may use the last relations to eliminate the generators $t_{ii}$ for $i=0,\dots, r$.

The operad in Lie algebras $\frt$ has a natural grading by weight, where we assign weight $1$ to all generators $t_{ij}$.\footnote{In fact, the usual convention is to assign weight 2 to all generators, bt this normalization is not important here.} 
We will also use the weight graded dual cyclic cooperad in coalgebras $\frtc$, and its Chevalley-Eilenberg complex
\[
C(\frtc) := (S(\frtc[-1]), d_{CE}).
\]
One has a quasi-isomorphism of dg Hopf cooperads
\[
    C(\frtc) \xrightarrow{\sim} \BV^c.
\]
It is defined on the commutative algebra generators by the composition 
\[
    \frtc[-1] \to \Q( t_{ij}^*) \xrightarrow{t_{ij}^*\mapsto \omega_{ij}}  \BV^c,
\]
where $t_{ij}^*$ is the cogenerator of $\frtc$ dual to $t_{ij}$ and $\omega_{ij}$ are the standard commutative algebra generators of $\BV^c$. 
 
Furthermore, we can build another cyclic operad from $\frt$.
Let $\PaP$ be the operad in sets of parenthesized permutations. Elements of $\PaP(r)$ are rooted planar binary trees with set of leaves $1,\dots,r$. Equivalently, such trees can be written as a parenthesized permutation such as 
\[
1(4((35)6)).    
\]
$\PaP$ naturally forms an operad in sets, the composition given by grafting trees. It is furthermore a cyclic operad: Any planar rooted binary with $r$ leaves can be seen as a non-rooted binary tree with $r+1$ leaves, and the permutation group $S_{r+1}$ acts by permuting the labels on the $r+1$ leaves. 

We also consider the pair groupoid $\tPaP$, which is the operad in groupoids with the objects of $\tPaP(r)$ being $\PaP(r)$, and exactly one morphism between any pair of objects.
$\tPaP$ is a cyclic operad in groupoids.

Next, we consider the complete universal enveloping algebra operad
\[
 U(\frt),
\]
which is a cyclic operad in complete Hopf algebras.
The operad of parenthesized ribbon chord diagrams $\PaRCD$ is the the cyclic operad in categories enriched in complete Hopf algebras given by the product of $\tPaP$ with $U(\ft)$.
Concretely, the objects of $\PaRCD(r)$ are the parenthesized permutations $\PaP(r)$, and the morphisms between any pair of objects are $U(\frt)$.

It can be seen that $\PaRCD$ is generated by the following three operations:
\begin{align*}
X =*\otimes 1&\in \PaRCD(2)((12), (21) ) 
&
t_{12}=*\otimes t_{12}&\in \PaRCD((2))
&
A=*\otimes 1 &\in \PaRCD(3)((12)3, 1(23)).
\end{align*}

We also define the (non-cyclic) suboperad $\ft\subset \frt$ by simply dropping the generators $t_{ii}$. Similarly, one my also define the operad $\PaCD$ analogously to $\PaRCD$, just replacing $\frt$ by $\ft$. $\PaCD$ is a model for the non-framed little disks operad.

\subsection{Grothendieck-Teichmüller group and action on chord diagrams}

\begin{defn}
The Grothendieck-Teichmüller group $\GRT = \Q^\times \ltimes \GRT_1$ consists of pairs $(\lambda, \Phi)$ with $\lambda\in \Q^\times$ and $\Phi\in \GRT_1$ a
group-like element of the complete free associative algebra in two variables $\Q\langle\langle x,y\rangle \rangle$ that satisfies the following relations.
\begin{align*}
\Phi(y,x)&= \Phi(x,y)^{-1} \\
\Phi(x,y)\Phi(y,-x-y)\Phi(-x,-y,x) &=1 \\
\Phi(t_{12}, t_{23}+t_{24})
\Phi(t_{13}+t_{23}, t_{34})&=
\Phi(t_{23}, t_{34})
\Phi(t_{12}+t_{13}, t_{24}+t_{34})
\Phi(t_{12}, t_{23}) \quad \in U\ft(4).
\end{align*}
The multiplication on $\GRT_1$ is defined such that $(\Phi\cdot\Phi')(x,y)=\Phi(x,y)\Phi'(X,\Phi^{-1}(x,y)y\Phi(x,y))$, and the action of $\Q^\times$ on $\GRT_1$ is such that $\lambda\cdot \Phi(x,y)=\Phi(\lambda x, \lambda y)$.
\end{defn}
\begin{prop}\label{prop:GRT action}
The Grothendieck-Teichmüller group $\GRT$ acts continuously on the cyclic operad in complete Hopf algebras $\PaRCD$ such that an element $(\lambda, \Phi)$ acts on the generators as follows:
\begin{align*}
X &\mapsto X \\
t_{12} &\mapsto \lambda t_{12} \\
A&\mapsto \Phi(t_{12}, t_{23})
\end{align*}
\end{prop}
\begin{proof}
It is well known that the formulas above yield a well-defined action of $\GRT$ on the non-cyclic operad $\PaRCD$ \cite{BarNatan}.
We just need to check that the action is compatible with the cyclic structure. This in turn is sufficient to check on the generators. Clearly, for the generators $X$ and $t_{12}$ the cyclic action is preserved.
The transposition $\tau = (01)$ acts on the generator $A$ as follows:
\begin{align*}
   A_{123} & \mapsto A_{231}^{-1}.
\end{align*}
Hence we need to check that 
\begin{align*}
\Phi &\stackrel{?}= \tau \Phi(t_{23}, t_{31})^{-1}
=
\Phi(t_{23}, t_{30})^{-1}
=
\Phi(t_{23}, -t_{31}-t_{32}-t_{33})^{-1}
=
\Phi(t_{23}, t_{12}+t_{11}+t_{22})^{-1}
=
\Phi(t_{23}, t_{12})^{-1}
\stackrel{!}=
\Phi(t_{12},t_{23})
.
\end{align*}
Here we used that the elements $t_{ii}$ are all central in $\frt$, and the $\frt$-relations. The last equality is the antisymmetry relation in the definition of $\GRT$.
\end{proof}

\subsection{Formality of the framed little disks operad and $\GRT$-action on $\BV^c$}
Recall from \cite{Frextended} Fresse's construction a rational homotopy theory for operads. He shows there is a Quillen adjunction 
\[
\Omega_\sharp \colon \sSet\Op\rightleftarrows (\dgHOpc)^{op} \colon \G   
\]
between the category of simplicial operads and the category of dg Hopf cooperads. The geometric realization functor $\G$ is just the arity-wise application of the standard geometric realization functor of rational homotopy theory.
 We remark that the construction of the above adjunction in \cite{Frextended} readily extends to the cyclic operad setting -- one may just replace operads by cyclic operads throughout. This yields an analogous adjunction 
 \[
    \Omega_\sharp \colon \sSet\cycOp\rightleftarrows (\cycHOpc)^{op} \colon \G,
 \]
with $\G$ again the arity-wise application of the standard geometric realization functor.

Now the link between the objects $\PaRCD$ and $\BV^c$ is that there are weak equivalences of simplicial cyclic operads
\[
N_\bullet\mathbb G \PaRCD \xrightarrow{\sim} 
N_\bullet\mathbb G U(\frt) \xrightarrow{\cong}
\gamma(\frt)
\xrightarrow{\sim} 
\MC_\bullet(\frt)
= \G(C(\frtc)) \simeq L\G (\BV^c).
\]
linking the nerve of the grouplike elements in $\PaRCD$ to the derived geometric realization of $\BV^c$.
The intermediate element $\gamma(\frt)$ is Getzler's nerve construction. We refer to the discussion in \cite[section 3]{CIW} and \cite[II.14, III.5.0]{Frbook} for more details.

On the other hand, the cyclic formality of the framed little disks operad (see \cite[section 3]{CIW}) states that $N_\bullet\mathbb G \PaRCD$ is a model for the rationalization of the framed little disks operad $\flD_2$.
Concretely, a simplicial model for the latter is given by the nerve of the parenthesized ribbon braids operad $\PaRB$, and one has a direct morphism 
\[
N\PaRB \to N_\bullet\mathbb G \PaRCD.
\]
modeling the canonical morphism $\flD_2\to \flD_2^\Q$, cf. again \cite{CIW}.

In any case, the action of $\GRT$ on $\PaRCD$ (considered for now as a non-cyclic operad) yields a morphism 
\begin{align*}
\GRT &\to \Map_{\sSet\Op}(\widehat{N_\bullet\PaRB}, N_\bullet\mathbb G \PaRCD)
\simeq 
\Map_{\sSet\Op}(\widehat{N_\bullet\PaRB}, \G(C(\frtc)) )
\\&\cong  \Map_{\dgHOpc}( C(\frtc), \Omega_\sharp(\widehat{N\PaRB}))
\simeq \Map_{\dgHOpc}^h(\BV^c, \BV^c),
\end{align*}

with $\widehat{N\PaRB}$ a cofibrant replacement of $\PaRB$. We used the adjunction for the second equality and the formality of the framed little disks operad for the last.
This induces the map $\GRT\simeq \pi_0\Aut_{\dgHOpc}^h(\BV^c)$.
Now given that the action of $\GRT$ on $\PaRCD$ respects the cyclic structure by Proposition \ref{prop:GRT action}, we can consider the above chain of maps in the categories of cyclic (co)operads, and we obtain:

\begin{cor}\label{cor:GRT action}
The map from $\GRT$ to $\pi_0\Aut_{\dgHOpc}^h(\BV^c)$ factorizes through the cyclic homotopy automorphisms 
$$
\GRT\to \pi_0\Aut_{\cycHOpc}^h(\BV^c)\to \pi_0\Aut_{\dgHOpc}^h(\BV^c).
$$
\end{cor}


\subsection{Endomorphisms of $\BV^{mod}$}
\begin{prop}\label{prop:BV rigid}
    \begin{enumerate}
\item Any dg Hopf $\BV^c$-comodule morphism $\phi:\BV^{c,mod}\to \BV^{c,mod}$ is the identity.
\item Any biderivation of the identity morphism $\BV^{c,mod}\to \BV^{c,mod}$ is trivial.
    \end{enumerate}
\end{prop}

Here a degree $k$ biderivation $\xi$ of a cooperadic comodule map $f:\MOp\to \NOp$ is a degree $k$ morphism of the underlying graded cyclic sequences such that $f+\epsilon\xi$ is a cooperadic comodule map after extending the ground ring to $\Q[\epsilon]/\epsilon^2$, with $\epsilon$ of degree $-k$.

\begin{proof}
We equivalently show the dual statements for the $\BV$-module $\BV^{mod}$.
\begin{enumerate}
    \item 
For the first statement, note that $\BV^{mod}$ is generated as a $\BV$-module by $1\in \BV((2))$.
Hence any endomorphism $\phi$ is uniquely determined by knowing $\phi(1)\in \BV((2))$. But by degree reasons we have 
\[
\phi(1) = \lambda 1    
\]
for some $\lambda\in \Q$. But since the morphism $\phi$ needs to preserve the coalgebra structure, in particular the counit, we necessarily have $\lambda=1$, so that $\phi=\mathit{id}$.
\item Similarly, any derivation $\xi$ of the identity map $\BV^{mod}\to \BV^{mod}$ is uniquely determined by the image $\xi(1)\in \BV((2))$.  
By the same argument as before, there is no degree zero biderivation.
The only other possibility (by degree reasons) is to have a biderivation of degree -1 so that $\xi(1)=\lambda \Delta$, with $\lambda\in \Q$ and $\Delta$ the degree 1 generator of $\BV((2))$, i.e., the BV operator.

Note that $\BV^{mod}$ is generated by $1$ as a $\BV$-module, but not freely. In particular, we have that the binary product $m\in \BV((3))$ is symmetric and hence 
\[
  (1-\sigma) \cdot (1 \circ_2 c) = 0  
\]
for any $\sigma\in S_3$, and $c\in \BV(2)$ the commutative product generator. The biderivation must respect (i.e., annihilate) this relation, so that 
\[
   \lambda  (1-\sigma) \cdot (\Delta \circ_2 c) \stackrel{!}= 0.  
\]
Let us use the basis $(E_{ij})_{1\leq i<j \leq r}\subset \BV((r))$ of degree one part of $\BV((r))$, see \cite[section 2.7]{WillwacherDBV}.
In particular, $\Delta=E_{12}$ and 
\[
    \Delta \circ_2 c= E_{12}+E_{13}.
\]
This is not symmetric, for example for $\sigma=(12)$ we have 
\[
\sigma \cdot (E_{12}+E_{13}) = E_{12}+E_{23}\neq  E_{12}+E_{13} .  
\]
Hence we necessarily have $\lambda=0$ so that $\xi=0$ is trivial.
\end{enumerate}
\end{proof}

Note that the proposition in particular implies that 
\[
\Map^h_{\dgHModc-\BV^c}(\BV^c, \BV^c) 
\cong 
\Aut^h_{\dgHModc-\BV^c}(\BV^c) 
\]
since every (derived) endomorphisms induces a morphism of Hopf comodules on the cohomology level, which by the proposition is the identity.

\subsection{Homotopy endomorphisms and a dg Lie algebra}

To compute the homotopy endomorphisms of $\BV^{c,mod}$ we need to pick a cofibrant and a fibrant replacement. For the cofibrant replacement we take the Chevalley-Eilenberg complex of the framed Drinfeld-Kohno cooperad in Lie coalgebras
\[
  C(\frtc) \xrightarrow{\sim}  \BV^{c,mod}.
\]
Cofibrancy of this object can be shown parallel to the analog result \cite[II.14.1.7]{Frbook} for the Chevalley-Eilenberg complex of the non-framed Drinfeld Kohno operad in Lie algebras. 
For the fibrant replacement we take the comodule $W$-construction revisited in Appendix \ref{app:W},
\[
    \BV^{c,mod}\xrightarrow{\sim} W\BV^{c,mod}.
\]
Here we just remark that the Hopf cooperadic $W$ construction assigns to a Hopf cooperad $\COp$ another cofibrant Hopf cooperad $W\COp$ with a quasi-isomorphism $\COp\to W\COp$. The cooperad $W\COp\cong \Bar \oW\COp$ is identified with the cobar construction of a 1-shifted dg operad $\oW\COp$, which in turn comes with a quasi-isomorphism $\Bar^c\COp\to \oW\COp$ from the cooperadic cobar construction of $\COp$.
Similarly, the Hopf comodule $W$ construction assigns to the right $\COp$-comodule $\MOp$ a cofibrant $W\COp$-comodule $W\MOp$, with a quasi-isomorphism of $W\COp$-comodules $\MOp\to W\MOp$. Furthermore, $W\MOp\cong \Bar\oW \MOp$ is identified with the comodule bar construction of the $\oW\COp$-module $\oW\MOp$, which in turn comes with a quasi-isomorphism of $\Bar^c\COp$-modules from the module cobar construction of $\MOp$, $\Bar^c\MOp\xrightarrow{\sim} \oW\MOp$.
The $W$ construction is hence a version of the bar-cobar resolution, that however preseves the Hopf (dgca) structure on objects, in contrast to the usual bar-cobar resolution.

Then the homotopy endomorphisms of $\BV^{c,mod}$ are, as a simplicial set
\[
\Map^h_{\dgHModc_{\BV^c}}(\BV^c, \BV^c) 
\simeq 
\Map_{\dgHModc_{W\BV^c}}(C(\ftc), W\BV^{c,mod}).
\]
The object $C(\frtc)$ is a collection of free dgcas, generated by $\frtc[-1]$, and the object $W\BV^{c,mod}$ is cofree as a $W\BV^c$-comodule, cogenerated by $\oW \BV^c$.
Since a morphism between a free and a cofree object is (usually) determined by restriction to generators and projection to cogenerators, it is natural to define the graded vector space
\begin{equation}\label{equ:gdef}
\fg := \iHom_{\bbS}(\frtc[-1], \oW \BV^c)
=
\prod_{r\geq 2} \iHom_{\bbS_r}\left(\frtc((r)), \oW \BV^c((r))\right)[1].  
\end{equation}
\begin{prop}\label{prop:g}
There is a filtered complete dg $\sLie$ algebra structure on $\fg$ with the following properties.
\begin{enumerate}
\item We have 
\[
\Map_{\dgHModc_{W\BV^c}}(C(\frtc), W\BV^{c,mod}) \simeq \MC_\bullet(\fg).
\] 
\item The Maurer-Cartan element $0\in \fg$ corresponds to the natural quasi-isomorphism of dg Hopf $W\BV^c$-comodules
\[
    C(\frtc) \to \BV^c \to W\BV^{c,mod}.
\]
\item The dg Lie algebra structure is compatible with the descending complete filtration on $\fg$ inherited from the weight grading on $\ftc$.
\item The differential $d$ on $\fg$ has the form 
\[
d = d_{\oW \BV^c}  + (\cdots)
\]
with $(\cdots)$ terms that strictly increase the weight or the arity (i.e., the $r$ in \eqref{equ:gdef}).
\end{enumerate}
\end{prop}

We will show the proposition in Appendix \ref{sec:def cx}. The proof is rather technical, following the constructions in \cite{FTW}. Furthermore, we shall use the following general criterion for dg Lie algebras to have a contractible Maurer-Cartan space.

\begin{lemma}\label{lem:simple GM}
Let $\fh$ be a dg $\sLie$ algebra equipped with a compatible descending complete filtration 
\[
\fh = \mF^1\fh \subset \mF^2\fh \subset \cdots    
\]
such that the cohomology of the associated graded satisfies
\[
H^k(\gr \fh) =0 
\]
for $k\leq 0$.
Then $\MC_\bullet(\fh)$ is weakly contractible, $\MC_\bullet(\fh)\simeq *$.
\end{lemma}
\begin{proof}
This can be seen as a special case of the Goldman-Millson Theorem \cite{DolRog}, for the map of dg Lie algebras $0\to \fh$.
However, since the conditions required in the formulation of this theorem in \cite{DolRog} do not quite match our setup, we will just reprove the result.

First one checks that $\MC_\bullet(\fh)$ is connected.
Equivalently, any MC element $x\in \MC_\bullet(\fh)$ is gauge trivial.
Suppose inductively that $x\in \mF^p\fh$.
Then the MC equation implies that $[x]\in \gr^p\fh$ is a degree zero cocycle and hence exact by assumption, $x=dy+\mF^{p+1}$.
Hence gauge transforming $x$ by $y$ yields an MC element $x'\in \mF^{p+1}\fh$.
Chaining the sequence of gauge transformations thus obtained as in \cite[]{DolRog} yields a gauge transformation sending $x$ to zero as desired.

Next we claim that $\pi_k\MC_\bullet(\fh)=0$ for all $k\geq 1$.
We have that $\pi_k\MC_\bullet(\fh)=H^{-k}(\fh)$, see \cite{Berglund}.
But the spectral sequence associated to our filtration converges to 
$H(\fh)$ and has zero first page in non-positive degrees by assumption. Hence in particular $H^{-k}(\fh)=0$ for all positive $k$ as desired.
\end{proof}

\subsection{Proof of Theorem \ref{thm:aut BVc contr}}
By Proposition \ref{prop:g} we have to check that 
\[
\MC_\bullet(\fg) \simeq *.    
\]
To show this we apply Lemma \ref{lem:simple GM} to our dg Lie algebra $\fg$, equipped with the weight filtration.
That is, we want to show that $H^k(\gr_W \fg)=0$ for $k\leq 0$.
To compute $H(\gr_W \fg)$ we endow $\gr_W\fg$ with a further descending complete filtration by arity and consider the associated spectral sequence.
By the third assertion of Proposition \ref{prop:g} the associated graded complex is 
\[
\gr_{ar}\gr_W \fg \cong (\fg, d_{\oW\BV^{c}}) 
=
(\iHom_{\bbS}(\frtc[-1], \oW\BV^{c,mod}), d_{\oW\BV^c}).
\]
The cohomology is 
\begin{equation}\label{equ:pre g 1}
    \iHom_{\bbS}(\frtc[-1], H(\oW\BV^{c,mod}))
\prod_{r\geq 2} \frt((r)) \otimes_{\bbS_r} H(\oW\BV^{c,mod}((r))) [1]
=
\prod_{r\geq 2} \frt((r)) \otimes_{\bbS_r} H(\Bar^c\BV^{c,mod}((r)))[1].
\end{equation}
For the last equality we used that $\oW\BV^{c,mod}$ is quasi-isomorphic to the comodule cobar construction, cf. Appendix \ref{app:W}.
We can compare the comodule and the cooperadic cobar construction by (the dual of) Proposition \ref{prop:MMPs}. We have
\begin{align*}
    H(\Bar^c\BV^{c,mod}) &= \meis\oplus H(\overline{\Bar^c\BV^{c,mod}}) 
    &&\text{and}&
    H(\overline{\Bar^c\BV^{c,mod}}) &=  \MMP^*\otimes H(\overline{\Bar^c\BV^c}),
\end{align*}
where $\meis$ is the one-dimensional symmetric sequence concentrated in degree 0 and arity 2, see \eqref{equ:meis def}.
Furthermore, by \cite[Theorem 2.21 and Proposition 3.9]{DCV} we have that
\[
    H(\overline{\Bar^c\BV^c}((r)))
    \cong
    \begin{cases}
        u\Q[u][1] & \text{ƒor $r=2$} \\
        H_c^\bullet(\M_{0,r}) & \text{ƒor $r\geq 3$}
    \end{cases}.
\]
Here $u$ is a formal variable of degree $+2$ and $H_c^\bullet(\M_{0,r})$ is the compactly supported cohomology of the moduli space of genus zero curves with $r$ marked points. 
Note that $H_c^k(\M_{0,r})=0$ for $k< r-3$, while $\frt$ is concentrated in degree 0, and $\MMP^*$ is concentrated in degree $+1$. 
Hence the only terms in \eqref{equ:pre g 1} contributing cohomology in non-positive degrees are the following.  
\begin{itemize}
\item For $r=2$, one has that $\frt((2))$ is the trivial $\bbS_2$-module concentrated in degree $0$, while $H(\Bar^c_{mod}\BV^c((2)))$ is concentrated in degrees $1,3,\dots$.
Hence the only contribution in non-positive degrees comes from the summand $\meis$ above. This yields a one-dimensional vector space spanned by $t_{12}\otimes 1$ in degree $-1$. 
\item For $r=3$ we have that $H_c^\bullet(\M_{0,3})=\Q$ is the trivial $\bbS_3$-representation concentrated in degree zero. 
Hence $H(\Bar^{c}\BV^{c,mod}((3)))$ is one copy of the irreducible $\bbS_3$-representation\footnote{We denote the irreducible representation of a symmetric group corresponding to a partition $\lambda$ by $V_\lambda$.} $V_{21}$ concentrated in degree +1.
On the other hand, $\frt((3))\cong \Q^3\cong V_3 \oplus V_{21}$  is concentrated in degree $0$. Hence the $r=3$-factor in \eqref{equ:pre g 1} is one-dimensional concentrated in degree 0, spanned by (for example) $\partial_1 \otimes t_{23}$.
\item For $r\geq 4$ we have that $H_c^\bullet(\M_{0,r})$ is concentrated in degrees $\geq 1$, hence $H(\Bar^c\BV^{c,mod}((r)))$ is concentrated in degrees $\geq 2$, and the corresponding factors in \eqref{equ:pre g 1} live in positive degrees.
\end{itemize}

We hence see that on the $E^1$ page of our inner spectral sequence, we only have one-dimensional cohomology in degree -1 (in weight $1$ and arity $r=2$), one-dimensional cohomology in degree 0 (in weight $1$ and arity $r=2$), and no other non-positive cohomology.

We are hence done if we can show that both classes cancel on the next page $E^2$ of our arity spectral sequence.
While this can be checked explicitly, it is easier to use the following abstract argument. 
Suppose they would not cancel.
Then the degree -1 class would survive in the cohomology of $\gr_W\fg$.
It could also not be canceled in the later pages of the weight spectral sequence by degree and weight reasons.
Hence we would have that $H^{-1}(\fg)=\Q$. This implies that the comodule $\BV^c$ has a nontrivial degree -1 biderivation.
Hence, passing to cohomology, $\BV^c$ has a degree -1 biderivation as well. It is nontrivial, since the cohomology representative $t_{12}\otimes 1$ above corresponds to the map that sends the co-BV operator to 1. But this then contradicts the second part of Proposition \ref{prop:BV rigid}.
Hence we are done. \hfill\qed

\subsection{Automorphisms of pointed operadic modules and pairs}
To compare pointed and non-pointed comodule automorphisms we will use the following lemma.
\begin{lemma}\label{lem:Hopf ptpair pair seq}
    Let $\mu:(\COp, \MOp)\to \peis^*$ and $\nu:(\DOp, \NOp)\to \peis^*$ be objects in $\pt\PairHOpc$. Then their derived mapping spaces in $\ptPairOp$ and $\PairOp$ fit into a homotopy fiber sequence 
    \[
    \Map^h_{\pt\PairHOpc}\left((\COp, \MOp), (\DOp, \NOp)\right)
    \to 
    \Map^h_{\PairHOpc}\left((\COp, \MOp), (\DOp, \NOp)\right)
    \to 
    \Map^h_{\bbS_2-\dgca}(\MOp((2)), \Q).
    \]
    Similarly, for $\MOp'\to \eis_{\COp}$ another pointed right $\COp$-comodule we have the homotopy fiber sequence 
    \[
        \Map^h_{\pt\dgHModc_{\COp}}(\MOp, \MOp')
        \to 
        \Map^h_{\dgHModc_{\COp}}(\MOp, \MOp')
        \to 
        \Map^h_{\bbS_2-\dgca}(\MOp((2)), \Q).
    \] 
\end{lemma}
\begin{proof}
    This follows from the following general result.
    Let $\cC$ be a model category and let $X\in \cC$ a fibrant object.
    Let $\pi: B\to X$ be fibrant in the overcategory $\cC_{/X}$, and let $\nu:A\to X$ be cofibrant in $\cC_{/X}$. Concretely, this means that $\iota$ is a fibration in $\cC$ and $A$ is a cofibrant object of $\cC$.
    Then the morphism 
    \[
    \Map_{\cC}(A, B)\to \Map_{\cC}(A, X)
    \]
    obtained by postcomposition with $\pi$ is an $\sSet$-fibration by \cite[Proposition II.3.2.12]{Frbook}.
    The fiber over the morphism $\nu:A\to X$ is 
    \[
        \Map_{\cC_{/X}}(\nu,\pi).
    \]
    Also note that $B$ is fibrant in $\cC$ since $\pi$ is a fibration. Hence all mapping spaces above model the respective derived mapping spaces.

    This general observation readily applies to the situation at hand, and noting that 
    \[
       \Map_{\PairHOpc}((\COp, \MOp), \peis^*) \cong 
       \Map_{\dgHModc_{\COp}}( \MOp, \meis^*) \cong\Map^h_{\bbS_2-\dgca}(\MOp((2)), \Q)
    \]
    the lemma follows.
\end{proof}

By Lemma \ref{lem:Hopf ptpair pair seq} we have a homotopy fiber sequence
\[
\Map^h_{\pt\dgHModc_{\BV^c}}(\BV^c, \BV^c)
\to
\Map^h_{\dgHModc_{\BV^c}}(\BV^c, \BV^c)
\to S^1_\mathbb{Q}.
\]
Since the middle space is weakly contractible by Theorem \ref{thm:aut BVc contr} the long exact sequence on homotopy groups gives us:

\begin{cor}\label{cor:map pt mod}
We have that $\Map^h_{pt\dgHModc_{\BV^c}}(\BV^c, \BV^c)\cong \Q$, that is,
\[
   \pi_k \Map^h_{pt\dgHModc_{\BV^c}}(\BV^c, \BV^c)
   =
   \begin{cases}
    \Q & \text{for $k=0$} \\
    * & \text{for $k\geq 1$}
   \end{cases}.
\]
\end{cor}

Similarly:
\begin{cor}\label{cor:map pair}
    We have that 
    \[
        \Aut^h_{\PairHOpc}(F(\BV^c))\cong \Aut^h_{\HOpc}(\BV^{c,nc}),
    \]
    that is,
    \[
       \pi_k \Map^h_{\PairHOpc}(\BV^c, \BV^c)
       =
       \begin{cases}
        \GRT & \text{for $k=0$} \\
        \Q & \text{for $k= 1$} \\
        * & \text{for $k\geq 1$}
       \end{cases}.
    \]
\end{cor}
\begin{proof}
    We consider the long exact sequence of homotopy groups associated to the fibration of Proposition \ref{prop:pair opc opc fibration}, using Theorem \ref{thm:aut BVc contr},
    \[
        \cdots
        \to 0
        \to \pi_1 \Aut^h_{\PairHOpc}(\BV^c)
        \to \Q
        \to 0
        \to \pi_0 \Aut^h_{\PairHOpc}(\BV^c) 
        \to \GRT. 
    \]
Clearly, we are done if we can check that the last arrow is surjective.
This surjectivity is equivalent to the statement that the derived $\GRT$-action on $\BV^{c,nc}$ extends to the pair $F(\BV^c)=(\BV^{c,nc},\BV^{c,mod})$. But this follows if the action respects the cyclic structure, by functoriality of $F$.
But this is Corollary \ref{cor:GRT action}, so that the final arrow in the sequence above is indeed onto.
\end{proof}

\subsection{Proof of Corollary \ref{cor:Aut BV}}
Finally we consider the homotopy automorphism space of the cyclic dg Hopf cooperad $\BV^c$.
By Theorem \ref{thm:main Hopf} we have that 
\[
\Aut^h_{\cycHOpc}(\BV^c) \simeq \Aut^h_{\pt\PairHOpc}(\BV^{c,nc}, \BV^{c,mod}).  
\]
By Lemma \ref{lem:Hopf ptpair pair seq} the right-hand side fits into a homotopy fiber sequence 
\[
    \Aut^h_{\pt\PairHOpc}(\BV^{c,nc}, \BV^{c,mod})
    \to
    \Aut^h_{\PairHOpc}(\BV^{c,nc}, \BV^{c,mod}) 
    \to
    S^1_\Q
\]
The homotopy groups of the middle space are computed in Corollary \ref{cor:map pair}. Hence we obtain the following long exact sequence of homotopy groups
\begin{equation}\label{equ:aut BV long exact}
    \cdots
    \to 0
    \to \pi_1 \Aut^h_{\cycHOpc}(\BV^c)
    \to \Q
    \xrightarrow{f} \Q
    \to \pi_0 \Aut^h_{\cycHOpc}(\BV^c) 
    \to \GRT
    \to 0,
\end{equation}
with all homotopy groups to the left being trivial.
We claim that the middle morphism 
\begin{equation}\label{equ:pre aut bv1}
f: \Q\cong \pi_1 \Aut^h_{\PairHOpc}(F(\BV^c))
\to 
\Q\cong \pi_1 \Map_{\bbS_2-\dgca}(\BV((2)), \Q)
\end{equation}
is an isomorphism. From this Corollary \ref{cor:Aut BV} clearly follows.

The map $f$ is induced by restricting derived automorphisms of $F(\BV^c)$ to the binary part of the comodule component, and then composing with the counit.
In particular, the left-hand side of \eqref{equ:pre aut bv1} comes from the $S^1$-action on the non-cyclic operad $\BV^{nc}$. This in turn is represented by the biderivation $\xi$ of $\BV^{nc}$ of degree $-1$ 
\[
\BV^{nc}(r)
\ni x \mapsto \xi(x):=
\Delta \circ_1 x - (-1)^{|x|}
\sum_{j=1}^r x\circ_j \Delta.     
\]
We first extend this derivation to the pair $(\BV^{nc}, \BV^{mod})$. This is done by the assignment 
\[
\BV^{mod}((r))
\ni y \mapsto \xi^{mod}(y):=
\pm\sum_{j=1}^r y\circ_{j,0} \Delta.     
\]
In particular, this means that 
\[
    \xi^{mod}(1) = \pm 2\Delta.
\]
This means that the morphism $f$ in \eqref{equ:aut BV long exact} is multiplication by $\pm 2$, and hence an isomorphism.
\hfill\qed

\appendix 
\section{Deformation complex for Hopf comodules}
\label{sec:def cx}

Operadic $W$ constructions are a fairly well-studied subject in the literature, and general accounts exist \cite{BM, BMColored, FTW}.
Similarly, deformation complexes for operads, Hopf cooperads and comodules have also been used in various other works.
Unfortunately, the precise technical situation we need for Proposition \ref{prop:g}, namely the version for Hopf comodules $\MOp$ over Hopf cooperads with nontrivial unary operations, does not readily exist in the literature. We hence briefly revisit these technical developments and remark how literature results can be extended to apply to our situation at hand.

\subsection{$W$ constructions}
\label{app:W}
\newcommand{\catT}{\mathcal{T}}
\newcommand{\vstar}{\mathrm{star}}
\newcommand{\vroot}{\mathit{root}}
Let $\COp$ be a dg Hopf cooperad. We allow that $\COp$ has non-trivial unary cooperations, but require that $\COp(1)$ is connected.
This means that the unary part $\bar \COp(1)$ of the coaugmentation coideal is concentrated in strictly positive degrees. Furthermore, let $\MOp$ be a right $\COp$ comodule

We then extend the cooperadic $W$ construction from \cite[section 5]{FTW} to this context.
For $S$ a set we consider a category $\catT_S$ whose set of objects are the rooted trees with set of leaves $S$ and at least bivalent vertices. 
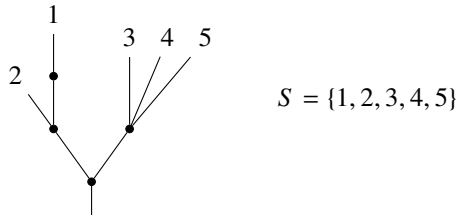
\[
\begin{tikzcd}
    \node[int] (r) at (0,-1) {};
    \node[int] (v1) at (-.5,-.3) {};
    \node[int] (v2) at (0.5,-.3) {};
    \node[int] (w) at (-.5,.4) {};
    \node (e1) at (-.5,1.1) { 1 };
    \node (e2) at (-1,.3) {2};
    \node (e3) at (.5,.8) {3};
    \node (e4) at (1,.8) {4};
    \node (e5) at (1.5,.8) {5};
    \draw (r) edge +(0,-.5) edge (v1) edge (v2)
    (v1) edge (w) edge (e2)
    (w) edge (e1)
    (v2) edge (e3) edge (e4) edge (e5);
\end{tikzcd}
\quad \quad 
S=\{1,2,3,4,5\}
\]
Here the valence of a vertex refers to the total valence, i.e., the number of children plus one. Note that in contrast to \cite[section 5]{FTW} we allow bivalent vertices. The morphisms of $\catT_S$ are generated by subgraph contraction. Also denote by $\catT_S'\subset \catT_S$ the full subcategory of trees that have at least one vertex.
(In other words, exclude the trivial tree consisting of the root connected to one edge, arising only for $|S|=1$.)

We then define three functors.
\begin{itemize}
\item The functor 
\begin{gather*}
    \bar \COp \colon \catT_S^{op} \to \dgca \\
    T \mapsto \otimes_T \bar \COp = \otimes_{v\in VT} \bar \COp(\vstar(v))
\end{gather*}
assigns to every tree the tree-wise tensor product.
Here $\vstar(v)$ is the set of children (including leaves) of the vertex $v$ in the vertex set $VT$ of $T$. The contraction morphism is sent to the cooperadic cocomposition.
\item The functor 
\begin{gather*}
    \MOp \colon (\catT_S')^{op} \to \dgca \\
    T \mapsto \otimes_T (\MOp, \bar \COp) = 
    \MOp(\vstar(\vroot))\otimes 
    \bigotimes_{v\in VT \atop v\neq \vroot} \bar \COp(\vstar(v))
\end{gather*}
assigns the tree-wise tensor product, with $\MOp$ decorating the bottom vertex $\vroot$, while all other vertices are decorated by $\bar \COp$.
The contraction morphisms are sent to either the coaction or the cocomposition, depending on whether the edge is incident to $\vroot$ or not. 
\item The functor 
\begin{gather*}
    E \colon \catT_S \to \dgca \\
    T \mapsto \otimes_{e\in ET} \Q[t,dt]
    \cong \Q[t_e,dt_e\mid e\in ET]
\end{gather*}
takes a tree $T$ to a dg commutative algebra that is a product of polynomial differential forms on the interval, with one factor for each internal edge of $T$. The morphism of $\catT_S$ contracting some edge $e$ is sent to the evaluation at $t_e=0=dt_e$.
\end{itemize}

We then define the ends 
\begin{align*}
    W\COp(S) &= \int_{T\in \catT_S} \bar \COp(T) \otimes E(T) \\
    W\MOp(S) &= \int_{T\in \catT_S'} \MOp(T) \otimes E(T).
\end{align*}
These assemble into symmetric sequences in dg commutative algebras $W\COp$, respectively $W\MOp$. 
Note that our category $\catT_S$ can have infinitely many objects, while the analogous object in \cite[section 5]{FTW} had only finitely many, due to requiring a trivalence condition on vertices.
However, if we fix a cohomological degree $k$ then the number of trees that contribute to the end $(W\COp(S))_k$ is still finite in our case.
This is because every bivalent vertex carries a decoration of degree at least 1, due to our connectedness assumption on $\COp(1)$, and all other contributing tensor factors are of non-negative degree.

Hence from this point on we may again apply the same line of arguments as in \cite[section 5.2]{FTW} and conclude the following:
\begin{itemize}
\item The object $W\COp$ is a dg Hopf cooperad, and comes with a natural weak equivalence (quasi-isomorphism) of dg Hopf cooperads 
\[
\COp\to W\COp.    
\]
The cooperadic cocomposition is by de-grafting of trees, i.e., cutting of an edge.
In the process the edge decorations of the cut edge $e$ is evaluated at $t_e=1$, $dt_e=0$.
\item As a (non-Hopf) dg cooperad we have that 
\[
W\COp \cong \Bar \oW\COp    
\]
is identified with the bar construction of an augmented 1-shifted dg operad $\oW\COp$. Concretely, we have that 
\[
    \overline{\oW\COp}(S) = \int_{T\in \catT_S'} \bar \COp(T) \otimes E_\partial(T),
\]
with $E_\partial(T)\subset E(T)$ the subspace of differential forms that vanish on the planes $\{t_e=1\}$ for all edges $e\in ET$.
The operadic composition is by grafting trees, decorating the newly formed edge $e$ by $dt_e$.
\item The 1-shifted dg operad $\oW\COp$ in turn is weakly equivalent to the dg operadic cobar construction  
\[
\Bar^c(\COp) \xrightarrow{\sim} \oW\COp.
\]
\item The object $W\COp$ is fibrant in $\dgHOpc$.
\end{itemize}

The arguments leading to these statements are also readily applicable to $W\MOp$ and yield:
\begin{itemize}
    \item The object $W\MOp$ is a dg Hopf $W\COp$-comodule, and comes with a natural weak equivalence (quasi-isomorphism) of dg Hopf $W\COp$-comodules
    \[
    \COp\to W\COp.    
    \]
    The cooperadic coaction is again by de-grafting of trees.
    \item As a (non-Hopf) dg comodule we have that 
    \[
    W\MOp \cong \Bar \oW\MOp    
    \]
    is identified with the bar construction of an $\oW\COp$-module $\oW\MOp$. Concretely, we have that 
    \[
        \oW\MOp(S) = \int_{T\in \catT_S'} \MOp(T) \otimes E_\partial(T).
    \]
    \item The module $\oW\MOp$ in turn is weakly equivalent to the dg operadic module cobar construction  
    \[
    \Bar^c(\MOp) \xrightarrow{\sim} \oW\MOp.
    \]
\item The object $W\MOp$ is fibrant in $\dgHModc_{W\COp}$.
    \end{itemize}

\subsection{The proof (sketch) of Proposition \ref{prop:g}}
We observe that our connectivity assumption for $\COp(1)$ above is satisfied in our case $\COp=\BV^c$. In particular, the $W$ construction $W\BV^{c,mod}$ is well-defined.

Next may use the proof of \cite[Proposition 7.3]{Wobstruction}.
That proposition is formulated for modules over reduced $\La$ Hopf cooperads. However, the arguments do not use the $\La$ structure and are applicable in our situation provided that the recursion in the proof of \cite[Lemma 7.1]{Wobstruction} converges after finitely many steps.
But this is certainly true in our case $\MOp=C(\frtc)$, since any application of the coproduct or coaction reduces arity or weight.
The conclusion of \cite[Proposition 7.3]{Wobstruction} is hence that we have a curved $\sLie_\infty$-structure on $\fg$ such that assertion a. of Proposition \ref{prop:g} holds.
But this means that in particular the canonical morphism of $W\BV^c$-comodules 
\[
C(\frtc) \to \BV^{c,mod} \to W\BV^{c,mod}
\]
corresponds to a Maurer-Cartan element $\alpha\in \fg$. Replacing the curved $\sLie_\infty$-structure on $\fg$ with its $\alpha$-twisted version we obtain a non-curved $\sLie_\infty$-structure on $\fg$ such that assertion b. of Proposition \ref{prop:g} also holds.

Furthermore, the construction of the $\sLie_\infty$-structure in \cite[Proposition 7.3]{Wobstruction} uses only natural operations, and in particular the Lie cobracket on $\frtc$ only once. It follows that the $\sLie_\infty$-structure cannot have homotopies of arity $\geq 3$, and is hence an honest dg Lie structure.
Furthermore, assertion c. of Proposition \ref{prop:g} follows since all operations that go into the construction of the dg Lie structure preserve the weight filtration.

For the last assertion d. of Proposition \ref{prop:g} one has to unpack the differential of the dg Lie structure on $\fg$ further.
We just note that from the construction of \cite[Proposition 7.3]{Wobstruction} it is clear that the differential on $\fg$ has the form $d_{\oW \BV^c}  + (\cdots)$, with $(\cdots)$ terms that are built using at least one application of the Lie coalgebra structure on $\frtc$, or a reduced cocomposition. Both these operations strictly reduce the weight, and/or the arity. Hence, on the space $\fg$ consisting of functions on $\fg$ these terms $(\cdots)$ increase weight or arity as claimed.

We also refer to (the more complicated) \cite[Proposition 23]{FWAut} for a similar statement on the differential of a Hopf cooperadic deformation complex.
\hfill\qed

\bibliographystyle{amsalpha}

\end{document}